\documentclass[10pt]{article}

\usepackage[T1]{fontenc} 

\usepackage[utf8]{inputenc} 

\usepackage{graphicx} 
\graphicspath{{fig/}} 

\usepackage{enumitem} 

\usepackage{subfig} 

\usepackage{amsmath,amssymb,amsthm} 

\usepackage{varioref} 

\usepackage{float} 

\usepackage{physics}

\usepackage{algorithm}

\usepackage[noend]{algpseudocode}

\usepackage{microtype} 

\usepackage[english]{babel} 

\usepackage{enumitem} 

\usepackage{varioref} 

\usepackage{booktabs} 

\usepackage{color}

\usepackage{bm}

\usepackage{mathtools}

\usepackage[page,header]{appendix}

\usepackage{titletoc}

\theoremstyle{plain} 
\newtheorem{theorem}{Theorem}[section]
\newtheorem{proposition}{Proposition}[section]
\newtheorem{lemma}[theorem]{Lemma}

\theoremstyle{definition} 

\theoremstyle{remark}
\newtheorem{remark}[theorem]{Remark}

\numberwithin{equation}{section}

\def\mO{\mathcal{O}}

\def\mQ{\mathcal{Q}}

\newcommand{\rd}{{\rm d}}
\newcommand{\R}{\mathbb{R}}
\newcommand{\bmv}{{\bm v}}

\newcommand{\bmj}{{\bm j}}
\newcommand{\bmk}{{\bm k}}

\newcommand{\bmxi}{{\bm \xi}}
\newcommand{\bmmu}{{\bm \mu}}
\newcommand{\bmsg}{{\bm \sg}}
\newcommand{\bmeta}{{\bm \eta}}
\newcommand{\bmzeta}{{\bm \zeta}}
\newcommand{\wtf}{\widetilde{f}}
\newcommand{\bmvs}{{\bm v_*}}

\newcommand{\dl}[2]{\frac{\mathrm{d} #1}{\mathrm{d} #2}}
\newcommand{\f}[2]{\frac{#1}{#2}}

\newcommand{\og}{\omega}

\newcommand{\dt}{\delta}

\newcommand{\nn}{\nonumber}
\newcommand{\ap}{\alpha}

\newcommand{\ld}{\lambda}

\newcommand{\tht}{\theta}
\newcommand{\ift}{\infty}

\newcommand{\ep}{\epsilon}

\newcommand{\Dt}{\Delta}

\newcommand{\sg}{\sigma}

\newcommand{\be}{\begin{equation}}
\newcommand{\ee}{\end{equation}}
\newcommand{\ba}{\begin{array}}
	\newcommand{\ea}{\end{array}}
\newcommand{\bea}{\begin{eqnarray}}
\newcommand{\eea}{\end{eqnarray}}
\newcommand{\beas}{\begin{eqnarray*}}
	\newcommand{\eeas}{\end{eqnarray*}}
\newcommand{\beit}{\begin{itemize}}
	\newcommand{\eit}{\end{itemize}}

\newcommand{\mL}{\mathcal{L}}

\def\mB{\mathcal{B}}
\newcommand{\bmK}{{\bm K}}
\newcommand{\bmz}{{\bm z}}
\newcommand{\bfi}{{\bf i}}

\begin{document}

\title{A fast Petrov-Galerkin spectral method for the multi-dimensional Boltzmann equation using mapped Chebyshev functions\footnote{JH and XH are partially supported by NSF CAREER grant DMS-1654152. JS is partially supported by NSF DMS-2012585 and AFOSR FA9550-20-1-0309. HY is partially supported by NSF CAREER grant DMS-1945029.}} 

\author{
Jingwei Hu\footnote{Department of Mathematics, Purdue University, West Lafayette, IN 47907, USA (jingweihu@purdue.edu).}  \  \ 
\and
Xiaodong Huang\footnote{Department of Mathematics, Purdue University, West Lafayette, IN 47907, USA (huan1178@purdue.edu).}  \  \ 
\and 
Jie Shen\footnote{Department of Mathematics, Purdue University, West Lafayette, IN 47907, USA (shen7@purdue.edu). }  \  \ 
\and 
Haizhao Yang\footnote{Department of Mathematics, Purdue University, West Lafayette, IN 47907, USA (yang1863@purdue.edu).}
}

\maketitle

\begin{abstract}
Numerical approximation of the Boltzmann equation presents a challenging problem due to its high-dimensional, nonlinear, and nonlocal collision operator. Among the deterministic methods, the Fourier-Galerkin spectral method stands out for its relative high accuracy and possibility of being accelerated by the fast Fourier transform. However, this method requires a domain truncation which is unphysical since the collision operator is defined in $\mathbb{R}^d$. In this paper, we introduce a Petrov-Galerkin spectral method for the Boltzmann equation in the unbounded domain. The basis functions (both test and trial functions) are carefully chosen mapped Chebyshev functions to obtain desired convergence and conservation properties. Furthermore, thanks to the close relationship of the Chebyshev functions and the Fourier cosine series, we are able to construct a fast algorithm with the help of the non-uniform fast Fourier transform (NUFFT). We demonstrate the superior accuracy of the proposed method in comparison to the Fourier spectral method through a series of 2D and 3D examples.
\end{abstract}




\section{Introduction}
\label{Sec:intro}

In multi-scale modeling, kinetic theory serves as a basic building block that bridges microscopic particle models and macroscopic continuum models. By tracking the probability density function, kinetic equations describe the non-equilibrium dynamics of the complex particle systems and have been widely used in disparate fields such as rarefied gas dynamics \cite{Cercignani00}, plasma physics \cite{BL}, nuclear reactor modeling \cite{Chandrasekhar}, chemistry \cite{Giovangigli}, biology, and socioeconomics \cite{NPT}. 

In this paper, we consider the numerical approximation of the Boltzmann equation, one of the fundamental equations in kinetic theory \cite{Cercignani, Villani02}. The complete equation include both particle transport and collisions which are often treated separately by operator splitting. Since the collision part is the main difficulty when numerically solving the equation, we focus on the following spatially homogeneous Boltzmann equation:
\be
   \partial_tf = Q (f,f),\quad
   t>0,\quad
   \bmv\in\R^d, \quad d=2,3,
   \label{Boltz_eqn}
\ee 
where $f=f(t,\bmv)$ is the probability density function of time $t$ and velocity $\bmv$, and $Q (f,f)$ is the collision operator whose bilinear form is given by
\be
   Q (g,f)(\bmv) = 
   \int_{\R^d} \int_{S^{d-1}}
    \mathcal{B} (\bmv,\bmvs,{\bm \sg})
   \left[g({\bm v'_*})f(\bmv')  -g({\bm v_*})f(\bmv) \right]\,
   {\rm d} {\bm \sg} \,{\rm d} \bmvs,
   \label{collison_operator}
\ee 
where the post-collisional velocities $(\bmv', {\bm v'_*})$ are defined in terms of the pre-collisional velocities $(\bmv, {\bm v_*})$ as
\be 
   \begin{cases}
   {\bm v'} = \f{1}{2} (\bmv+\bmvs) + \f{1}{2} |\bmv-\bmvs| {\bm \sg}, \vspace{3pt}\\
   {\bm v'_*} = \f{1}{2} (\bmv+\bmvs) - \f{1}{2} |\bmv-\bmvs| {\bm \sg},
   \end{cases}
\ee 
with ${\bm \sg}$ being a vector over the unit sphere $S^{d-1}$. The collision kernel $\mathcal{B}$ takes the form
\be
   \mathcal{B}(\bmv,\bmv_*,{\bm \sg})=B(|\bmv - \bmv_*|,\cos \tht), \quad
   \cos \tht = \left< \f{\bmv - \bmv_*}{|\bmv - \bmv_*|}, {\bm \sg} \right>,
   \label{collision_form}
\ee 
i.e., it is a function depending only on the relative velocity $|\bmv - \bmv_*|$ and cosine of the scattering angle. The collision operator $Q(f,f)$ satisfies many important physical properties, including conservation of mass, momentum, and energy:
\begin{equation} \label{consv}
\int_{\mathbb{R}^d}Q(f,f)\,\rd{\bmv}=\int_{\mathbb{R}^d}Q(f,f)\bmv\,\rd{\bmv}=\int_{\mathbb{R}^d}Q(f,f)|\bmv|^2\,\rd{\bmv}=0,
\end{equation}
and the Boltzmann's H-theorem:
\begin{equation}
\int_{\mathbb{R}^d}Q(f,f)\log f\,\rd{\bmv}\leq 0.
\end{equation}

In the physically relevant case ($d=3$), the collision operator is a five-fold quadratic integral whose numerical approximation can be extremely challenging. The stochastic methods, such as the direct simulation Monte Carlo (DSMC) methods proposed by Nanbu \cite{Nanbu80} and Bird \cite{Bird}, have been historically popular due their simplicity and efficiency. However, like any Monte Carlo methods, they suffer from slow convergence and high statistical noise especially for low-speed and unsteady flows. In the past two decades, the deterministic methods have undergone extensive development largely due to the advance in computing powers, see \cite{Pareschi} for a recent review.

Among the deterministic methods for the Boltzmann equation, the Fourier-Galerkin spectral method stands out for its relative high accuracy and possibility of being accelerated by the fast Fourier transform (see, for instance, \cite{PR00, MP06, GT09, GHHH17} for major algorithmic development and \cite{PR00stability, FM11, AGT18, HQY21} for stability and convergence analysis). Although being a method with reasonable efficiency and accuracy tradeoff, the Fourier spectral method requires a domain truncation which is unphysical since the original collision operator is defined in $\mathbb{R}^d$. This truncation changes the structure of the equation and often comes with an accuracy loss. Inspired by the recent work \cite{hu2020petrov} of the two authors here where a spectral method was introduced for the 1D inelastic Boltzmann equation, we develop in this paper a Petrov-Galerkin spectral method for the Boltzmann equation  (\ref{Boltz_eqn})\footnote{Unlike the inelastic Boltzmann equation which has a non-trivial solution in 1D, the classical Boltzmann equation (\ref{Boltz_eqn}) must be considered at least for $d\geq 2$.} using mapped Chebyshev functions in $\mathbb{R}^d$. Both the test functions and trial functions are carefully chosen to obtain desired approximation properties. Furthermore, thanks to the close relationship of the Chebyshev functions and the Fourier cosine series, we are able to construct a fast algorithm with the help of the non-uniform fast Fourier transform (NUFFT). This speedup is critical as the direct implementation of the proposed method would require excessive storage for precomputation and significant online computational cost that soon become a bottleneck for larger $N$ (the number of spectral modes). Extensive numerical tests in 2D and 3D are performed to demonstrate the accuracy and efficiency of the proposed method. In particular, the comparison with the Fourier spectral method in \cite{GHHH17} indeed confirms the better approximation properties of the proposed method. Finally, we mention some recent spectral methods \cite{fonn2014polar, gamba2018galerkin, kitzler2019polynomial, HC20, HCW20} for the Boltzmann equation that use other orthogonal polynomial bases in $\mathbb{R}^d$. Up to our knowledge, our method is the first with a fast implementation and a consistency analysis. 

The rest of this paper is organized as follows. In Section~\ref{Sec:MCspectral}, we introduce the mapped Chebyshev functions in $\mathbb{R}^d$ along with their approximation properties. In Section~\ref{Sec:PGMCFspectral}, we construct the Petrov-Galerkin spectral method for the Boltzmann equation using the mapped Chebyshev functions as trial and test functions. The approximation properties for the collision operator and moments are proved as well. The numerical realization including the fast algorithm is described in detail in Section~\ref{Sec:num_realiz}. In Section~\ref{Sec:num}, several numerical tests in 2D and 3D are performed to demonstrate the accuracy and efficiency of the proposed method. The paper is concluded in Section~\ref{Sec:con}.

\section{Multi-dimensional mapped Chebyshev functions}
\label{Sec:MCspectral}

In this section, we introduce the mapped Chebyshev functions in $\mathbb{R}^d$ and discuss their approximation properties. These functions are extension of the one-dimensional mapped Chebyshev functions introduced in \cite{hu2020petrov} based on tensor product formulation \cite{shen2011spectral,shen2014approximations}. Later in Section~\ref{Sec:PGMCFspectral}, they will serve as the trial functions and test functions in the Petrov-Galerkin spectral method for the Boltzmann equation.

\subsection{Mapped Chebyshev functions in $\mathbb{R}^d$}
\label{Subsec:Basis_function}

To define the mapped Chebyshev functions in $\mathbb{R}^d$, we start with the one-dimensional Chebyshev polynomials  on the interval $I=(-1,1)$:
\be
T_0(\xi) = 1,\quad T_1(\xi) = \xi, \quad T_{k+1}(\xi) = 2\xi T_k(\xi) - T_{k-1}(\xi), \quad k\geq 1.
\ee
Define the inner product $(\cdot,\cdot)_{\omega}$ as
\begin{equation}
(F,G)_{\omega}:=\int_I F(\xi)G(\xi)\omega(\xi)\, \rd{\xi}, \quad \omega(\xi)=(1-\xi^2)^{-\frac{1}{2}},
\end{equation}
then $\{T_k(\xi)\}_{k\geq 0}$ satisfy the orthogonality condition
\be 
   (T_k,T_l )_{\og} = c_k\dt_{k,l}, \quad \forall \ k, l \geq 0,
\ee
where $c_0=\pi$ and $c_k=\pi/2$ for $k\geq 1$.

We then introduce a one-to-one mapping $\xi \rightarrow v(\xi)$ (its inverse is denoted as $v \rightarrow \xi(v)$) from $I$ to $\mathbb{R}$ such that
\be 
   \f{{\rm d} v}{{\rm d}\xi} = \f{S}{(1-\xi^2)^{1+\frac{r}{2}}} := \f{\og(\xi)}{[\mu(\xi)]^2}, \quad
 v(\pm 1)= \pm \ift, 
   \label{1D_mapping}
\ee
where $S>0$ is a scaling parameter, $r\geq0$ is the tail parameter, and the function $\mu$ is given by
\be
   \mu(\xi) = \f{(1-\xi^2)^{\f{1+r}{4}}}{\sqrt{S}}.
     \label{1D_mu_function}
\ee
With this mapping we define two sets of mapped Chebyshev functions in $\R$ as
\be  \label{1DCheby}
  \widetilde{T}_k(v) := \f{[\mu(\xi(v))]^{4}}{\sqrt{c_k}} T_k(\xi(v)), \quad \widehat{T}_k(v) := \f{[\mu(\xi(v))]^{-2}}{\sqrt{c_k}} T_k(\xi(v)).
\ee
Define the inner product $(\cdot,\cdot)_{\mathbb{R}}$ as
\begin{equation}
(f,g)_{\mathbb{R}}:=\int_{\mathbb{R}} f(v)g(v)\, \rd{v},
\end{equation}
then it is easy to check that $\{ \widetilde{T}_k(v)\}_{k\geq 0}$ and $\{ \widehat{T}_k(v)\}_{k\geq 0}$ satisfy the orthonormal condition:
\begin{equation}
 (\widetilde{T}_k, \widehat{T}_l )_{\mathbb{R}}=\delta_{k,l}, \quad \forall \ k, l \geq 0.
\end{equation}

\begin{remark}
   Two one-to-one mappings between $I$ and $\R$ often used in practice are 
   \begin{itemize}
      \item {\bf logarithmic mapping} $(r = 0)$:
      \be \label{logarithmic}
         v= \f{S}{2} \ln \left(\f{1+\xi}{1-\xi}\right),\quad
         \xi = \tanh\left(\f{v}{S}\right),\quad 
         \mu(\xi) = \f{1}{\sqrt{S}}(1-\xi^2)^{\frac{1}{4}},
      \ee 
      \item {\bf algebraic mapping} $(r = 1)$:
      \be \label{algebraic}
         v = \f{S\xi}{\sqrt{1-\xi^2}},\quad
         \xi = \f{v}{\sqrt{S^2+v^2}},\quad 
         \mu(\xi) = \f{1}{\sqrt{S}}(1-\xi^2)^{\frac{1}{2}}.
      \ee 
   \end{itemize}
\end{remark}

In the multi-dimensional case, we denote the multi-vector as ${\bm v}=(v_1,\dots,v_d)$ and multi-index as $\bmk= (k_1,\ldots,k_d)$, where $k_j$ is a non-negative integer for each $j=1,\dots, d$; $0\le \bmk \le N$ means $0\le k_j \le N$ for each $j=1,\dots, d$. We define the mapped Chebyshev functions in $\mathbb{R}^d$ using (\ref{1DCheby}) via the tensor product as
\be 
   \widetilde{\bm T}_\bmk({\bmv}) := \prod_{j=1}^d \widetilde{T}_{k_j}(v_j), \quad    \widehat{\bm T}_\bmk(\bmv) := \prod_{j=1}^d \widehat{T}_{k_j}(v_j).
\ee
The inner products $(\cdot,\cdot)_{\bm \og}$ in $I^d=(-1,1)^d$ and $(\cdot,\cdot)_{\R^d}$ in $\R^d$ are defined, respectively, by
\be
   (F,G)_{\bm \og} := \int_{I^d} F({\bm \xi}) G({\bm \xi})  {\bm \og}({\bm \xi})\, {\rm d}{\bm \xi},\quad
   (f,g)_{\R^d} := \int_{\R^d} f(\bmv) g(\bmv)\, {\rm d}{\bm v},
\ee
with the weight function ${\bm \og}({\bm \xi}):=\prod_{j=1}^d \omega(\xi_j)$. 
Then we have
\be  \label{ortho}
   \left(\widetilde{\bm T}_\bmk,\widehat{\bm T}_{\bm l} \right)_{\R^d} 
   = \prod_{j=1}^d \left(\widetilde{T}_{k_j},\widehat{T}_{l_j}\right)_{\R} 
   =\prod_{j=1}^d \delta_{k_j,l_j}
   =: {\bm \dt}_{\bmk,{\bm l}} .
\ee 
Therefore, any $d$-variate function $f(\bmv)$ can be expanded using $\{\widetilde{\bm T}_\bmk(\bmv)\}_{\bmk \geq 0}$ as
\be 
   f(\bmv) 
      = \sum_{\bmk\geq 0} \widetilde{f}_\bmk \widetilde{\bm T}_\bmk(\bmv)
      = \sum_{\bmk\geq 0} \widetilde{f}_\bmk \frac{[{\bm \mu}({\bm \xi})]^{4}}{\sqrt{\bm{c}_\bmk}}{\bm T}_\bmk({\bm \xi}),
\ee 
and the expansion coefficients $\{\widetilde{f}_\bmk\}_{\bmk\geq 0}$ are determined by 
\bea 
\widetilde{f}_\bmk 
= \left(f, \widehat{\bm T}_\bmk \right)_{\R^d} 
= \f{1}{\sqrt{\bm{c}_\bmk}} \left([{\bm \mu}({\bm \xi})]^{-4}f(\bmv({\bm \xi})), {\bm T}_\bmk ({\bm \xi})\right)_{\bm \og},
\label{MultiD_spectral_coef} 
\eea 
where
\be 
   {\bm \mu} ({\bm \xi}) := \prod_{j=1}^d \mu(\xi_j), \quad
   \bm{c}_\bmk := \prod_{j=1}^{d} c_{k_j}, \quad
   {\bm T}_\bmk({\bm \xi}) := \prod_{j=1}^{d} T_{k_j} (\xi_j),
   \label{multiD_mu_c_Tk}
\ee 
and $\bmv({\bm \xi})$ is the mapping from $I^d$ to $\mathbb{R}^d$ such that each component $\xi_j$ is mapped to $v_j$ via the 1D mapping (\ref{1D_mapping}). The inverse mapping ${\bm \xi}(\bmv)$ is understood similarly.

In Section~\ref{Sec:PGMCFspectral}, we will introduce the Petrov-Galerkin spectral method for the Boltzmann equation in $\mathbb{R}^d$, where the trial function space and test function space are chosen, respectively, as 
\begin{equation}
\widetilde{\mathbb{T}}_N^d := \{\widetilde{\bm T}_\bmk(\bmv)\}_{0\le \bmk\le N}, \quad \widehat{\mathbb{T}}_N^d := \{\widehat{\bm T}_\bmk(\bmv)\}_{0\le \bmk\le N}.
\end{equation}
The choice of these functions is motivated by their decay/growth properties at large $|\bmv|$. The following 
result  is a straightforward extension of the 1D result in \cite{hu2020petrov}.
\begin{lemma}
   For any $\bm{k}\ge 0$ and $|\bmv| \gg 1$, we have
   \bea 
      \left|\widetilde{\bm T}_\bmk({\bm v})\right| 
      \sim  \begin{cases}
                  e^{-\frac{2}{S}(\sum_{j=1}^d |v_j|)}, & r=0, \\
                  \prod_{j=1}^d |v_j|^{-4}, & r=1;
         \end{cases}\quad
      \left|\widehat{\bm T}_\bmk ({\bm v}) \right| 
      \sim  \begin{cases}
                    e^{\frac{1}{S}(\sum_{j=1}^d |v_j|)}, & r=0, \\
                  \prod_{j=1}^d |v_j|^{2}, & r=1,
         \end{cases}
   \eea 
where $r=0$ corresponds to the logarithmic mapping (\ref{logarithmic}) and $r=1$ to the algebraic mapping (\ref{algebraic}).
\end{lemma}


\subsection{Approximation properties}
\label{subsec:approximation}

We describe below some approximation properties of the mapped Chebyshev functions in $\mathbb{R}^d$. 

For a function $f(\bm v)$ defined in $\mathbb{R}^d$, the transform $\bmv(\bmxi)$ maps it to a function in $I^d$. Hence, we introduce the linked function pair $(f,F)$ such that $f(\bmv)=f(\bmv(\bm \xi))\equiv F(\bm \xi)$. In addition, we introduce another function pair $(\widehat{f}^{\ap}, \widehat{F}^{\ap})$ as
\be
   \widehat{f}^{\ap} ({\bmv})
   := {f}({\bm v}) [\bm{\mu}(\bmxi(\bm{v}))]^{-\ap}
   = {F}(\bmxi) [\bm{\mu}(\bmxi)]^{-\ap}
   =: \widehat{F}^{\ap} (\bmxi).
\ee
We define the approximation space in $\mathbb{R}^d$ with a parameter $\ap$ as
\be 
   \mathbb{V}_N^{\ap,d}(\mathbb{R}^d):= \text{span}\left\{{\bm T}_\bmk^{\ap}({\bm v}) 
   := [{\bm \mu}({\bm \xi}({\bm v}))]^{\ap} {\bm T}_\bmk ({\bm \xi}({\bm v})),\ 0\le \bmk\le N \right\}. 
\ee 
Therefore, the trial function space $\widetilde{\mathbb{T}}_N^d$ and test function space $\widehat{\mathbb{T}}_N^d$ introduced in the previous section correspond to $ \mathbb{V}_N^{4,d}$ and $ \mathbb{V}_N^{-2,d}$, respectively. 

In the following, the $L^2$ space with a given weight ${\bm w}$ is equipped with norm
\begin{equation}
\|f\|_{L_{\bm w}^2(I^d)}=\left(\int_{I^d}|f(\bmxi)|^2{\bm w}(\bmxi)\,\rd{\bmxi}\right)^{1/2} \  \text{ or }  \  \|f\|_{L_{\bm w}^2(\R^d)}=\left(\int_{\R^d}|f(\bmv)|^2{\bm w}(\bmv) \,\rd{\bmv}\right)^{1/2},
\end{equation}
depending on the domain of interest.

Let $\mathbb{P}_N^d(I^d)$ denote the set of $d$-variate polynomials in $I^d$ with degree $\leq N$ in each direction, and $\Pi_N^{d}: L_{{\bm \og}}^2 (I^d)\to \mathbb{P}_N^d(I^d)$ be the Chebyshev orthogonal projection operator such that
\be
   \left( \Pi_N^{d} {F} - {F}, {\phi} \right)_{{\bm \og}}=0,\quad
   \forall {\phi} \in {\mathbb P}_N^d(I^d).
\ee
Then we define another projection operator $\pi_N^{\ap,d}: L_{\bm{\mu}^{2-2\ap}}^2 ({\mathbb R}^d) \to {\mathbb V}_N^{\ap,d}(\mathbb{R}^d)$ by
\be
   \pi_N^{\ap,d} f
   :=
   \bm{\mu}^{\ap} \Pi_N^{d} (F\bm{\mu}^{-\ap} )
   =
   \bm{\mu}^{\ap} \Pi_N^{d} \widehat{F}^{\ap}.
   \label{MCFs_projector}
\ee
One can verify using the definition that
\bea
   \left(\pi_N^{\ap,d}f - f,\bm{\mu}^{2-2\ap} \bm{T}_\bmk^{\ap}
   \right)_{\R^d}
   &=& \int_{\R^d} 
   (\pi_N^{\ap,d}f - f) 
  \bm{\mu}^{2-\ap}  \bm{T}_\bmk (\bmxi(\bm{v})) 
   \, {\rm d} {\bm v}\nn \\
   &=& \int_{I^d}
   \left[\bm{\mu}^{\ap} \Pi_N^{d} \widehat{F}^{\ap} - \bm{\mu}^{\ap} \widehat{F}^{\ap} \right]
   {\bm T}_\bmk ({\bm \xi})
   \bm{\mu}^{2-\ap} 
   \f{\bm{\og}({\bm \xi})}{\bm{\mu}^2}
   \, {\rm d} {\bm \xi} \nn \\
   &=& \left(\Pi_N^{d} \widehat{F}^{\ap}-\widehat{F}^{\ap}, {\bm T}_\bmk \right)_{\bm \og} = 0,\quad
   \forall \ 0\le \bmk \le N.
   \label{projection_verified}
\eea

Next, we introduce the function space $\bm{B}_{\ap}^m (\R^d)$ equipped with the norm 
\be
   \left\| f \right\|_{\bm{B}_{\ap}^m (\R^d)}
   =\left(\sum_{0\le \bm{k} \le m}
   \left\| {\bm D}_{\ap, \bm{v}}^\bmk f\right\|_{L_{\bm{\varpi}^{\bm{k}+\frac{1+r}{2}{\bf 1}}}^2 ({\mathbb R}^d)}^2
   \right)^{1/2},
\ee
and semi-norm
\be
   \left| f \right|_{\bm{B}_{\ap}^m (\R^d)}
   = \left(\sum_{j=1}^d 
   \left\|D_{\ap, v_j}^m f \right\|_{L_{\bm{\varpi}^{m \bm{e}_j+\frac{1+r}{2}{\bf 1}}}^2 ({\mathbb R}^d)}^2
   \right)^{1/2},
\ee
where $\bf 1$ is an all-one vector, $\bm{e}_j =(0,\dots,1,\dots,0)$ with $1$ in the $j$-th position and $0$ elsewhere, and
\be 
   \bm{D}_{\ap,\bm v}^\bmk f :=
   D_{\ap, v_1}^{k_1} \cdots D_{\ap, v_d}^{k_d} f,\quad
   \bm{\varpi}^{\bmk} 
   := \prod_{j=1}^d 
   (1-\xi(v_j)^2)^{k_j },
\ee
with
\be
D_{\ap , v_j}^{k_j} f
:= \underbrace{a(v_j) \frac{\partial}{\partial v_j}   
      \left( a(v_j) \frac{\partial}{\partial v_j} \left(
      \ldots \left(a(v_j)\frac{\partial\widehat{f}^{\ap}}{\partial v_j}\right) \ldots
      \right)\right)
   }_{k_j \text{ times derivatives}} 
   =\frac{\partial^{k_j} \widehat{F}^{\ap}}{\partial \xi_j} ,
\ee
where $a(v_j) := \dl{v_j}{\xi_j}$ is determined by the mapping.

We have the following approximation result.
\begin{theorem}
\label{Md_OMCF_approx}
  Let $\ap \in \R$, $r \ge 0$. If $ f\in \bm{B}_{\ap}^m (\R^d)$, we have 
   \be
      \left\| \pi_N^{\ap, d} f - f \right\|_{L_{\bm{\mu}^{2-2\ap}}^2 ({\mathbb R}^d)}
      \le
      C N^{-m} \left| f\right|_{\bm{B}_{\ap}^m(\R^d)}.
   \ee
\end{theorem}

\begin{proof}
Note that
\bea 
   \left\| \pi_N^{\ap, d} f - f \right\|_{L_{\bm{\mu}^{2-2\ap}}^2 ({\mathbb R}^d)}^2
   &=& \int_{\R^d} 
   (\pi_N^{\ap,d} f - f)^2
   \bm{\mu} ^{2-2\ap} 
   \, {\rm d} {\bm v}\nn \\
   &=& \int_{I^d}
   \left[\bm{\mu}^{\ap} \Pi_N^{d} \widehat{F}^{\ap} - \bm{\mu}^{\ap} \widehat{F}^{\ap}\right]^2
   \bm{\mu}^{2-2\ap} 
   \f{\bm{\og}({\bm \xi})}{\bm{\mu}^2}
   \, {\rm d} {\bm \xi} \nn \\
     &=& \left\|\Pi_N^{d} \widehat{F}^{\ap} - \widehat{F}^{\ap}\right\|_{L_{\bm{\og}}^2 (I^d)}^2.\nn
\eea

By the multi-variate (full tensor product) Chebyshev approximation result (Theorem 2.1 in \cite{shen2010sparse}), we know
\be
   \left\|\Pi_N^{d} \widehat{F}^{\ap} - \widehat{F}^{\ap} \right\|_{L_{ {\bm \og}}^2(I^d)} 
   \le 
   C N^{-m} \left(\sum_{j=1}^d 
   \left\|\partial_{\xi_j}^m \widehat{F}^{\ap}\right\|_{L_{{\bm \varpi}^{m \bm{e}_j-\frac{1}{2}{\bf 1}}}^2 (I^d)}^2
   \right)^{1/2}. \nn
\ee
Hence,
\bea
   && \left\| \pi_N^{\ap, d} f - f\right\|_{L_{\bm{\mu}^{2-2\ap}}^2 ({\mathbb R}^d)} 
   =\left\|\Pi_N^{d} \widehat{F}^{\ap} - \widehat{F}^{\ap}\right\|_{L_{\bm{\og}}^2 (I^d)} \nn \\
   &\le& 
   C N^{-m} \left(\sum_{j=1}^d 
   \left\|\partial_{\xi_j}^m \widehat{F}^{\ap}\right\|_{L_{\bm{\varpi}^{m \bm{e}_j-\frac{1}{2}{\bf 1}}}^2 (I^d)}^2 \right)^{1/2} \nn \\
  &\leq&
   CN^{-m} \left(\sum_{j=1}^d 
   \left\|D_{\ap, v_j}^m f \right\|_{L_{\bm{\varpi}^{m \bm{e}_j+\frac{1+r}{2}{\bf 1}}}^2 ({\mathbb R}^d)}^2
   \right)^{1/2}\nn\\
   &=& 
   C N^{-m} \left| {f}\right|_{\bm{B}_{\ap}^m (\R^d)}. \nn
\eea
\end{proof}

\section{A Petrov-Galerkin spectral method for the Boltzmann equation}
\label{Sec:PGMCFspectral}

We consider the initial value problem
\be
   \begin{cases}
    \partial_t f(t, \bmv)= Q (f,f), & t>0,  \quad \bmv\in \R^d, \\
      f(0, \bmv) = f^0 (\bmv),
   \end{cases}
   \label{MultiDBoltz}
\ee 
where $Q(f,f)$, in a strong form,  is given by (\ref{collison_operator}). To construct the Petrov-Galerkin spectral method, the following weak form of the collision operator is more convenient:
\be  \label{weak}
   \left( Q(f,f), \phi \right)_{\R^d}= \int_{\mathbb{R}^d} Q(f,f)(\bmv)\phi(\bmv)\,\rd{\bmv}=
      \int_{\R^d} \int_{\R^d} \int_{S^{d-1}}
   \mathcal{B} (\bmv,\bmvs,{\bm \sg})
   f(\bmv) f (\bmvs) [\phi(\bmv') -\phi(\bmv)] \, \rd{\bmsg} \, \rd{\bmv} \, \rd{\bmvs},
\ee 
where $\phi(\bmv)$ is a test function.

We look for an approximation of $f$ in the trial function space $\widetilde{\mathbb{T}}_N^d$ as
\be
   f(t, \bmv)
   \approx
   f_N(t, \bmv) 
   = \sum_{0\le \bmk\le N} \wtf_\bmk(t)
   \widetilde{\bm T}_\bmk(\bmv) \in \widetilde{\mathbb{T}}_N^d.
   \label{fN_Expansion}
\ee
Substituting $f_N$ into (\ref{MultiDBoltz}) and requiring the residue of the equation to be orthogonal to the test function space $\widehat{\mathbb{T}}_N^d$, we obtain
\be \label{petrov}
   \left(\partial_t{f_N} - Q(f_N,f_N),  \widehat{\bm T}_\bmk\right)_{\R^d} = 0 \quad \text{for all }\  \widehat{\bm T}_\bmk\in \widehat{\mathbb{T}}_N^d.
\ee 
By the orthogonality condition (\ref{ortho}),  we find that the coefficients $\{\wtf_\bmk(t)\}$ satisfy the following ODE system
\be
\begin{cases}
\f{\rm d}{{\rm d}t} \wtf_\bmk(t)
= \mQ_{\bm{k}}^N , \\
\wtf_\bmk(0) = 
\wtf_\bmk^0,
\end{cases}
\quad 0 \le \bmk \le N,
\label{spectral_ODE}
\ee
where 
\bea
   \mQ_{\bm{k}}^N 
   &:=&
   \left(Q (f_N, f_N), \widehat{\bm T}_\bmk\right)_{\R^d}
   \nn\\
   &=&
   \int_{\R^d} \int_{\R^d} \int_{S^{d-1}}
   \mathcal{B} (\bmv,\bmvs,{\bm \sg})
   f_N(\bmv) f_N(\bmvs) 
  \left[\widehat{\bm T}_\bmk(\bmv') - \widehat{\bm T}_\bmk(\bmv)\right] 
    \rd{\bmsg} \, \rd{\bmv} \, \rd{\bmvs},
   \label{Qk_component}
\eea
and
\be
\wtf_\bmk^0 := 
\left(f^0, \widehat{\bm T}_\bmk \right)_{\R^d}=\f{1}{\sqrt{\bm{c}_\bmk}} \left([{\bm \mu}({\bm \xi})]^{-4}f^0({\bmv}({\bm \xi})), {\bm T}_\bmk({\bm \xi})\right)_{\bm \og}.
\ee
Note that we used the weak form (\ref{weak}) in (\ref{Qk_component}).

\begin{remark}
An equivalent way of writing the ODE system (\ref{spectral_ODE}) is
\be
   \begin{cases}
    \partial_tf_N(t, \bmv) = 
    \pi_N^{4,d} Q (f_N, f_N), \\
     f_N(0, \bmv) = \pi_N^{4,d} f^0 (\bmv),
   \end{cases}
   \label{P_G_method}
\ee
where $\pi_N^{4,d}$ is the projection operator defined in (\ref{MCFs_projector}) (with $\alpha=4$). Indeed, for any $ f \in L_{\bm{\mu}^{-6}}^2 ({\mathbb R}^d)$,
\begin{equation}
\pi_N^{4,d} f= \sum_{0\le \bmk\le N} \left(f, \widehat{\bm T}_\bmk \right)_{\R^d} \widetilde{\bm T}_\bmk(\bmv) \in {\mathbb V}_N^{4,d}(\mathbb{R}^d)=\widetilde{\mathbb{T}}_N^d.
\end{equation}
\end{remark}

\subsection{Approximation property for the collision operator}

In this subsection, we establish a consistency result of the spectral approximation for the collision operator. We will show that if $f$ and $Q(f,f)$ have certain regularity, the proposed approximation of the collision operator $ \pi_N^{4,d} Q (\pi_N^{4,d}f, \pi_N^{4,d}f)$ enjoys spectral accuracy. We will only prove this result under the algebraic mapping (\ref{algebraic}) with $S=1$, that is in 1D,
\be
   v = \f{\xi}{\sqrt{1-\xi^2}},\quad
   \xi=\frac{v}{\sqrt{1+v^2}}, \quad
   \mu =\sqrt{1-\xi^2}= \f{1}{\sqrt{1+v^2}}.
\ee 
The reason of this choice is strongly motivated by the existing regularity result of the Boltzmann collision operator under an exponentially weighted Lebesgue norm:
\be \label{expLp}
   \|f\|_{\mL_{k}^p(\R^d)} =
   \left(\int_{\R^d} |f(\bmv)|^p (1+|\bmv|^2)^{kp/2}
   \, \rd \bmv\right)^{1/p}, \quad k \in \R, \ 1\leq p <\infty.
\ee 

Specifically, we write the collision operator (\ref{collison_operator}) as $Q (g,f) = Q^+ (g,f)-Q^- (g,f)$, where the gain part and loss part are given by
\begin{equation}
\begin{split}
 &Q^+(g,f)(\bmv) = 
   \int_{\R^d} \int_{S^{d-1}}
    \mathcal{B} (\bmv,\bmvs,{\bm \sg})
   g({\bm v'_*})f(\bmv')  \, \rd{\bmsg}  \, \rd{\bmvs},\\
& Q^-(g,f)(\bmv) = 
   \int_{\R^d} \int_{S^{d-1}}
    \mathcal{B} (\bmv,\bmvs,{\bm \sg})
   g({\bmvs})f(\bmv)\, \rd{\bmsg}  \, \rd{\bmvs}.
   \end{split}
\end{equation}
Then we have the following regularity result for the gain operator $Q^+ (g,f)$ established in \cite{mouhot2004regularity}.

\begin{theorem} [Theorem 2.1, \cite{mouhot2004regularity}] \label{theorem_gain} Let $k, \eta \in \mathbb{R}$, $1\leq p <\infty$, and let the collision kernel $\mathcal{B}$ satisfy certain cut-off assumption\footnote{To avoid technicality, we do not spell out the condition here and only mention that most of the collision kernels used in numerical simulations satisfy this assumption.}. Then the following estimate holds
\be
   \| Q^+ (g,f)\|_{\mL_{\eta}^p (\R^d)}
   \le 
   C_{k,\eta,p}(\mathcal{B}) 
   \|g\|_{\mL_{|k+\eta|+|\eta|}^1 (\R^d)} 
   \|f\|_{\mL_{k+\eta}^p (\R^d)}, 
   \label{general_gain_estimate}
\ee
where $C_{k,\eta,p}(\mathcal{B})$ is a constant that depends only on the kernel $\mathcal{B}$ and $k$, $\eta$ and $p$.
\end{theorem}

To obtain a similar estimate for the loss operator $Q^-(g,f)$, we restrict ourselves to the variable hard sphere (VHS) collision model \cite{bird1994molecular}. Note that this kernel falls into the assumption in Theorem~\ref{theorem_gain}. We have the following result.

\begin{proposition}
Let $\eta \in \mathbb{R}$, $1\leq p <\infty$, and let the collision kernel takes the form $\mathcal{B} = C_{\ld} |\bmv-\bmvs|^{\ld}$, where $0\le \ld \le 1$ and $C_{\ld}$ is a positive constant. Then the following estimate holds
\be 
   \|Q^- (g,f) \|_{\mL_{\eta}^p (\R^d)}
   \le
   C_{\lambda}
   \|g\|_{\mL_{\ld}^1 (\R^d)}  \|f\|_{L_{\ld+\eta}^p (\R^d)}. \label{loss term}
\ee
\end{proposition}

\begin{proof}
Note that 
\be
   |\bmv-\bmvs| 
   \le |\bmv| + |\bmvs|
   = (|\bmv|^2 + |\bmvs|^2 + 2|\bmv||\bm{v_*}|)^{1/2}
   \le \left(1+|\bmv|^2\right)^{1/2} \left(1+|\bmvs|^2\right)^{1/2}. \nn
\ee 
Then
\bea
   Q^- (g,f) (\bmv)
   &=& 
   C_{\ld} f(\bmv)
   \int_{\R^d} g(\bmvs) |\bm{v-v_*}|^{\ld}
   \, {\rm d} \bmvs
   \nn \\ 
   &\le&
   C_{\ld} f(\bmv) \left(1+|\bmv|^2\right)^{\ld/2}
   \left[\int_{\R^d} g(\bmv_*)\left(1+|\bmvs|^2\right)^{\ld/2}
   \, {\rm d} \bmvs\right]
   \nn\\
   &=&
   C_{\ld}
   f(\bmv) \left(1+|\bmv|^2\right)^{\ld/2} \|g\|_{L_{\ld}^1(\R^d) }. \nn
\eea
Therefore, for any $\eta \in \R$, $1\leq p <\infty$,
\be 
   \|Q^- (g,f) \|_{\mL_{\eta}^p (\R^d)}
   \le
   C_{\lambda}
   \|g\|_{\mL_{\ld}^1 (\R^d)}  \|f\|_{L_{\ld+\eta}^p (\R^d)}.\nn
\ee
\end{proof}

Combining the previous two results, we can obtain the following theorem.

\begin{theorem} \label{regularityQ}
Let the collision kernel takes the form $\mathcal{B} = C_{\ld} |\bmv-\bmvs|^{\ld}$, where $0\le \ld \le 1$ and $C_{\ld}$ is a positive constant. Then the collision operator $Q(g,f)$ satisfies
\be 
   \| Q (g,f)\|_{\mL_{3d}^2 (\R^d)}
   \le 
   C_d(\mathcal{B})
   \|g\|_{\mL_{\ld+6d}^1 (\R^d)}  \|f\|_{\mL_{\ld+3d}^2 (\R^d)},
\ee
where $C_d(\mathcal{B})$ is a constant that depends only on the kernel $\mathcal{B}$ and the dimension $d$.
\end{theorem}

\begin{proof}
Choosing $k=\lambda$, $\eta=3d$, $p=2$ in (\ref{general_gain_estimate}), we have
\be
   \| Q^+ (g,f)\|_{\mL_{3d}^2 (\R^d)}
   \le 
C_{\lambda,3d,2}(\mathcal{B})
   \|g\|_{\mL_{\lambda+6d}^1 (\R^d)} 
   \|f\|_{\mL_{\lambda+3d}^2 (\R^d)}. \nn
\ee
Choosing $\eta=3d$, $p=2$ in (\ref{loss term}), we have
\be 
   \|Q^- (g,f) \|_{\mL_{3d}^2 (\R^d)}
   \le
C_{\lambda}
   \|g\|_{\mL_{\ld}^1 (\R^d)}  \|f\|_{\mL_{\ld+3d}^2 (\R^d)}. \nn
\ee
Combining both, we obtain
\be
   \| Q (g,f)\|_{\mL_{3d}^2 (\R^d)}
   \le 
   C_d(\mathcal{B})
   \|g\|_{\mL_{\ld+6d}^1 (\R^d)}  \|f\|_{\mL_{\ld+3d}^2 (\R^d)}.\nn
\ee
\end{proof}

Before we proceed to the consistency proof, we need the following lemmas.
\begin{lemma}
\label{norm_inequality}
   Under the algebraic mapping (\ref{algebraic}) with $S=1$, we have
   \be
     \| f \|_{\mL_{\eta}^2(\R^d) }
      \le \| f \|_{L_{\bm{\mu}^{-2\eta}}^2(\R^d) } \le \| f \|_{\mL_{d\eta}^2(\R^d) } \quad \text{for any } \eta\geq 0.
   \ee 
\end{lemma}

\begin{proof}
Note that 
\begin{equation}
1+\sum_{j=1}^d |v_j|^2  \leq \prod_{j=1}^d (1+|v_j|^2)\leq (1+\sum_{j=1}^d |v_j|^2)^d. \nn
\end{equation}
Then we have
\bea
   \|f\|_{L_{\bm{\mu}^{-2\eta}}^2 (\R^d)} 
   &=& 
   \left(\int_{\R^d} |f(\bmv)|^2 \bm{\mu}^{-2\eta} 
   \, \rd \bmv\right)^{1/2}
   \nn \\
   &=& 
   \left(\int_{\R^d} |f(\bmv)|^2 \prod_{j=1}^d (1+|v_j|^2)^{\eta} 
   \ \rd \bmv\right)^{1/2} 
   \nn\\
   &\ge &
   \left(\int_{\R^d} |f(\bmv)|^2 \left(1+\sum_{j=1}^d |v_j|^2\right)^{\eta} 
   \, \rd \bmv\right)^{1/2}
   = \|f\|_{\mL_{\eta}^2 (\R^d)}. \nn
\eea 
Also,
\bea
   \|f\|_{L_{\bm{\mu}^{-2\eta}}^2 (\R^d)} 
   &=& 
   \left(\int_{\R^d} |f(\bmv)|^2 \prod_{j=1}^d (1+|v_j|^2)^{\eta} 
   \, \rd \bmv\right)^{1/2} 
   \nn\\
  &\le&   \left(\int_{\R^d} |f(\bmv)|^2 \left(1+\sum_{j=1}^d |v_j|^2\right)^{d\eta} 
   \, \rd \bmv\right)^{1/2}
   = \|f\|_{\mL_{d\eta}^2 (\R^d)}. \nn
   \eea
\end{proof}

\begin{lemma}
\label{norm_inequality_1}
   For any $\eta \geq 0$, there exist $\epsilon>0$ and $C_{\epsilon}>0$ such that
   \be
     \|f\|_{\mL_{\eta}^1 (\R^d)}
      \le C_{\epsilon}\|f\|_{\mL_{\eta+\frac{1+\epsilon}{2}}^2 (\R^d)}.
   \ee 
\end{lemma}

\begin{proof}
Note that
\bea
  \|f\|^2_{\mL_{\eta}^1 (\R^d)}&=& \left(\int_{\mathbb{R}^d} |f(\bmv)| (1+|\bmv|^2)^{\frac\eta 2}\, \rd{\bmv}\right)^2 \nn\\
  &\le & \int_{\mathbb{R}^d} |f(\bmv)|^2 (1+|\bmv|^2)^{\eta+\frac{1+\epsilon}{2}}\, \rd{\bmv}\,
   \int_{\mathbb{R}^d}  (1+|\bmv|^2)^{-\frac{1+\epsilon}{2}} \,\rd{\bmv} \nn \\
   &\le& C_\epsilon  \int_{\mathbb{R}^d} |f(\bmv)|^2 (1+|\bmv|^2)^{\eta+\frac{1+\epsilon}{2}} \,\rd{\bmv} \nn\\
   &=& C_{\epsilon}\|f\|^2_{\mL_{\eta+\frac{1+\epsilon}{2}}^2 (\R^d)}, \nn
\eea
where we used the Cauchy-Schwarz inequality.
\end{proof}

We are ready to present a consistency result. 

\begin{theorem} \label{consistency}
Let the collision kernel takes the form $\mathcal{B} = C_{\ld} |\bmv-\bmvs|^{\ld}$, where $0\le \ld \le 1$ and $C_{\ld}$ is a positive constant. Then under the algebraic mapping (\ref{algebraic}) with $S=1$, we have
\begin{equation}
\begin{split}
&  \left\|Q (f,f) - \pi^{4,d}_N Q (\pi_N^{4,d} f,\pi_N^{4,d} f) \right\|_{\mL_{3}^2 (\R^d)}\\
 \le &   C_{d,\epsilon}(\mathcal{B}) N^{-m}\left(  \big|f\big|_{{\bf B}^m_{\lambda+6d+\frac{3+\epsilon}{2}} (\R^d)} \|f\|_{\mL_{\lambda+3d}^2 (\R^d)}
   + \big|f\big|_{{\bf B}^m_{\lambda+3d+1} (\R^d)} \|f\|_{\mL^1_{\lambda+6d} (\R^d)}+ \big|Q ( f, f)\big|_{{\bm B}_4^m (\R^d)}\right),
\end{split}
\end{equation}
where $m$ is a positive integer, $d$ is the dimension, $\epsilon>0$ is a constant, and $C_{d,\epsilon}(\mathcal{B})$ is a constant depending only on the kernel $\mathcal{B}$, $d$ and $\epsilon$.
\end{theorem}

\begin{proof}

By the triangle inequality
\bea
   && \|Q (f,f) - \pi^{4,d}_N Q (\pi_N^{4,d} f,\pi_N^{4,d} f) \|_{\mL_{3}^2 (\R^d)}
   \nn\\
   &\le &
   \|Q (f,f) - \pi^{4,d}_N Q (f, f)\|_{\mL_{3}^2 (\R^d)} +
   \|\pi^{4,d}_N Q (f, f)- \pi^{4,d}_N Q (\pi_N^{4,d} f,\pi_N^{4,d} f) \|_{\mL_{3}^2 (\R^d)}.
   \label{consistency_triangle}\nn
\eea

For the first term, by Lemma~\ref{norm_inequality} and Theorem~\ref{Md_OMCF_approx}, we have
\bea
     \|Q (f,f) - \pi^{4,d}_N Q (f, f)\|_{\mL_{3}^2 (\R^d)} \le    \big\|Q (f,f) - \pi^{4,d}_N Q (f, f) \big\|_{L_{\bm{\mu}^{-6}}^2 (\R^d)} \le 
       C N^{-m} 
   \big|Q ( f, f)\big|_{{\bm B}_4^m (\R^d)}.
   \label{Consistency_part2}\nn
\eea

For the second term, using again Lemma~\ref{norm_inequality}, Theorem~\ref{Md_OMCF_approx}, and Lemma~\ref{norm_inequality_1}, Theorem~\ref{regularityQ}, we have
\begin{equation}
\begin{split}
 & \|\pi^{4,d}_N Q (f, f)- \pi^{4,d}_N Q (\pi_N^{4,d} f,\pi_N^{4,d} f)\|_{\mL_{3}^2 (\R^d)} \leq \|\pi^{4,d}_N Q (f, f)- \pi^{4,d}_N Q (\pi_N^{4,d} f,\pi_N^{4,d} f)\|_{L_{\bmmu^{-6}}^2 (\R^d)}
   \\
   \le&
     \| Q (f, f)-  Q (\pi_N^{4,d} f,\pi_N^{4,d} f)\|_{L_{\bmmu^{-6}}^2 (\R^d)}  \leq  \|Q (f, f)- Q (\pi_N^{4,d} f,\pi_N^{4,d} f)\|_{\mL_{3d}^2 (\R^d)} \\
   \leq &
   \|Q (f-\pi_N^{4,d} f,f) \|_{\mL_{3d}^2 (\R^d)}
   +
   \|Q (\pi_N^{4,d} f,f-\pi_N^{4,d} f)\|_{\mL_{3d}^2 (\R^d)}
   \nn\\
 \le &
   C_d(\mathcal{B}) 
   \left(
      \|f-\pi_N^{4,d} f\|_{\mL_{\lambda+6d}^1 (\R^d)}   
   \|f\|_{\mL_{\lambda+3d}^2 (\R^d)}
   +
   \|\pi_N^{4,d} f\|_{\mL_{\lambda+6d}^1 (\R^d)}   
   \|f-\pi_N^{4,d} f\|_{\mL_{\lambda+3d}^2 (\R^d)} \right) \nn\\
 \le &
  C_d(\mathcal{B}) 
   \left( \|f-\pi_N^{4,d} f\|_{\mL_{\lambda+6d}^1 (\R^d)}
         (\|f\|_{\mL_{\lambda+3d}^2 (\R^d)}
   +\|f-\pi_N^{4,d} f\|_{\mL_{\lambda+3d}^2 (\R^d)})+
   \| f\|_{\mL_{\lambda+6d}^1 (\R^d)}   
   \|f-\pi_N^{4,d} f\|_{\mL_{\lambda+3d}^2 (\R^d)} \right) \nn\\
 \le &
   C_{d,\epsilon}(\mathcal{B})   \|f-\pi_N^{4,d} f\|_{\mL_{\lambda+6d+\frac{1+\epsilon}{2}}^2 (\R^d)}\left( \|f\|_{\mL_{\lambda+3d}^2 (\R^d)}+\|f-\pi_N^{4,d} f\|_{\mL_{\lambda+3d}^2 (\R^d)}\right)
  \\
  & + C_d(\mathcal{B})\|f\|_{\mL^1_{\lambda+6d} (\R^d)}  \|f-\pi_N^{4,d} f\|_{\mL_{\lambda+3d}^2 (\R^d)} \\
 \le &
   C_{d,\epsilon}(\mathcal{B})  \|f-\pi_N^{4,d} f\|_{L_{{\bm{\mu}}^{-2(\lambda+6d+\frac{1+\epsilon}{2})}}^2 (\R^d)}( \|f\|_{\mL_{\lambda+3d}^2 (\R^d)}+\|f-\pi_N^{4,d} f\|_{L_{{\bm{\mu}}^{-2(\lambda+3d)}}^2(\R^d)})\\
  & + C_d(\mathcal{B})\|f\|_{\mL^1_{\lambda+6d} (\R^d)} \|f-\pi_N^{4,d} f\|_{L_{{\bm{\mu}}^{-2(\lambda+3d)}}^2(\R^d)} \\
\le &
   C_{d,\epsilon}(\mathcal{B}) N^{-m} \big|f\big|_{{\bf B}^m_{\lambda+6d+\frac{3+\epsilon}{2}} (\R^d)}\left(   \|f\|_{\mL_{\lambda+3d}^2 (\R^d)}
   + CN^{-m}  \big|f\big|_{{\bf B}^m_{\lambda+3d+1} (\R^d)} \right)\\
   & + C_d(\mathcal{B})N^{-m}\|f\|_{\mL^1_{\lambda+6d} (\R^d)}\big|f\big|_{{\bf B}^m_{\lambda+3d+1} (\R^d)} \\
 \le & C_{d,\epsilon}(\mathcal{B}) N^{-m} \left(\big|f\big|_{{\bf B}^m_{\lambda+6d+\frac{3+\epsilon}{2}} (\R^d)} \|f\|_{\mL_{\lambda+3d}^2 (\R^d)}+\big|f\big|_{{\bf B}^m_{\lambda+3d+1} (\R^d)}  \|f\|_{\mL^1_{\lambda+6d} (\R^d)} \right).
 \end{split}
 \end{equation}
 Combining the above inequalities, we arrive at the desired result.
  \end{proof}

\subsection{Approximation property for the moments}

For the Boltzmann equation, the moments or macroscopic observables are important physical quantities. Still, under the algebraic mapping (\ref{algebraic}), we can show that the spectral method (\ref{spectral_ODE}) preserves mass and energy.
\begin{theorem}
   If using the algebraic mapping (\ref{algebraic}) with $N\geq 2$, the spectral method (\ref{spectral_ODE}) preserves mass and energy, i.e., $\rho(t)$ and $E(t)$ defined by
   \bea \label{def_mom}
     \rho(t):= \int_{\R^d} f_N(t,\bmv) \, \rd \bmv, \quad E(t)=\int_{\R^d}  f_N(t,\bmv) |\bmv|^2\, \rd \bmv
   \eea
 remain constant in time. Furthermore,
 \begin{equation}
 \rho(t)\equiv \int_{\R^d} f^0(\bmv) \, \rd \bmv, \quad E(t)\equiv \int_{\R^d} f^0(\bmv)|\bmv|^2\,\rd{\bmv}.
 \end{equation}
\end{theorem}

\begin{proof}
In 1D, the first few Chebyshev polynomials read 
 \begin{equation*}
 T_0(\xi)=1, \quad T_1(\xi)=\xi, \quad T_2(\xi)=2\xi^2-1.
 \end{equation*} 
 With the algebraic mapping (\ref{algebraic}), we have
 \begin{equation*}
 T_0(\xi(v))=1, \quad T_1(\xi(v))=\frac{v}{\sqrt{v^2+S^2}}, \quad T_2(\xi(v))=\frac{v^2-S^2}{v^2+S^2}, \quad \mu(\xi(v))=\frac{\sqrt{S}}{\sqrt{v^2+S^2}}.
 \end{equation*} 
 Then
 \begin{equation*}
 \widehat{T}_k(v) = \f{[\mu(\xi(v))]^{-2}}{\sqrt{c_k}} T_k(\xi(v))=\frac{v^2+S^2}{\sqrt{c_k}S}T_k(\xi(v)).
 \end{equation*}  
Specifically,
\begin{equation*} 
 \widehat{T}_0(v)=\frac{v^2+S^2}{\sqrt{c_0}S}, \quad  \widehat{T}_1(v)=\frac{v\sqrt{v^2+S^2}}{\sqrt{c_1}S}, \quad \widehat{T}_2(v)=\frac{v^2-S^2}{\sqrt{c_2}S}.
 \end{equation*} 
Therefore,
 \begin{equation*}
1=\frac{\sqrt{c_0}}{2S} \widehat{T}_0(v)-\frac{\sqrt{c_2}}{2S}\widehat{T}_2(v), \quad v^2=\frac{\sqrt{c_0}S}{2}\widehat{T}_0(v)+\frac{\sqrt{c_2}S}{2}\widehat{T}_2(v).
\end{equation*}
Hence we can replace $(\widehat{T}_0(v), \widehat{T}_2(v))$  by  $(1,v^2)$ as basis functions, namely, 
 $$\widehat{\mathbb{T}}_N^1=\text{span}\{1,\widehat{T}_1,v^2,\widehat{T}_3,\widehat{T}_4,\cdots,\widehat{T}_N\}.$$
In $d$ dimensions, it is easy to see 
\begin{equation*} 
1, \ v_1^2, \ v_2^2, \ \dots, \ v_d^2 \in \widehat{\mathbb{T}}_N^d \quad \text{for } N\geq 2.
\end{equation*}
In other words, we have shown that $1, |\bmv|^2\in \widehat{\mathbb{T}}_N^d$ for $N\geq 2$.

On the other hand, by \eqref{def_mom} and \eqref{petrov}, we have
\begin{equation*}
 \begin{split}
     \frac{\rd }{\rd{t}} \rho(t)=&   (\partial_t f_N(t,\bmv),1)_{\R^d}=\left(Q (f_N, f_N), 1 \right)_{\R^d}=0;\\
       \frac{\rd }{\rd{t}} E(t)=&   (\partial_t f_N(t,\bmv),|\bmv|^2)_{\R^d}=\left(Q (f_N, f_N), |\bmv|^2 \right)_{\R^d}=0,
\end{split}
\end{equation*}
where in the last equality we used the conservation property (\ref{consv}) of the collision operator.

It remains to show
\begin{equation*}
\int_{\R^d} f_N(0,\bmv) \, \rd \bmv=\int_{\R^d} f^0(\bmv) \, \rd \bmv, \quad \int_{\R^d} f_N(0,\bmv)|\bmv|^2 \, \rd \bmv=\int_{\R^d} f^0(\bmv)|\bmv|^2 \, \rd \bmv.
\end{equation*}
Note that $f_N(0,\bmv)=\pi^{4,d}f^0$, it suffices to show
\begin{equation*}
(\pi^{4,d}f^0-f^0,1)_{\mathbb{R}^d}=(\pi^{4,d}f^0-f^0,|\bmv|^2)_{\mathbb{R}^d}=0,
\end{equation*}
which is true by (\ref{projection_verified}) (with $\alpha=4$).
\end{proof}

\section{Numerical realization}
\label{Sec:num_realiz}

To implement the proposed spectral method, one needs to solve the ODE system (\ref{spectral_ODE}). For time discretization, one can just use the explicit Runge-Kutta methods. Hence, the key is the efficient evaluation of $\mQ_{\bm{k}}^N$ as defined in (\ref{Qk_component}). 

In this section, we introduce two algorithms to compute $\mQ_{\bm{k}}^N$. The first one is a direct algorithm that treats $\mQ_{\bm{k}}^N$ as a matrix/tensor-vector multiplication. Since the weight matrix/tensor does not depend on the numerical solution $f_N$, it can be precomputed and stored for repeated use. This approach is simple but will soon meet a bottleneck when $N$ increases since the memory requirement as well as the online computational cost can get extremely high. To alleviate this, we propose a fast algorithm, where the key idea is to recognize the gain term of the collision operator as a non-uniform discrete Fourier cosine transform to be accelerated by the non-uniform FFT (NUFFT). Note that this is possible because we are using the mapped Chebyshev functions as a basis, which is related to the Fourier cosine series.

\subsection{A direct algorithm}

To derive the direct algorithm, we substitute (\ref{fN_Expansion}) into (\ref{Qk_component}) to obtain
\bea
   \mQ_{\bm{k}}^N 
   &=& 
   \sum_{0\le \bm{i,j}\le N} 
   \wtf_{\bm i} \wtf_{\bm j}
   \int_{\R^d} \int_{\R^d} \int_{S^{d-1}}
   \mathcal{B} (\bmv,\bmvs,{\bm \sg})\,
   \widetilde{\bm T}_{\bm i}(\bmv) \widetilde{\bm T}_{\bm j}(\bmvs)
   [\widehat{\bm T}_\bmk(\bmv') - \widehat{\bm T}_\bmk(\bmv)] 
   \, {\rm d} {\bm \sg} \, {\rm d} \bmv\, {\rm d} \bmvs
   \nn \\
   &=&
   \sum_{0\le \bm{i,j}\le N} 
   \wtf_{\bm i} \wtf_{\bm j}
   \left[\widetilde{I}_1(\bm{i,j,k}) - \widetilde{I}_2(\bm{i,j,k})\right], 
   \quad 
   0 \le \bmk \le N,
   \label{direct_online}
\eea
where
\bea
   \widetilde{I}_1(\bm{i,j,k})
   &:=& 
   \int_{\R^d} \int_{\R^d} \int_{S^{d-1}}
  \mathcal{B} (\bmv,\bmvs,{\bm \sg})\,
   \widetilde{\bm T}_{\bm i}(\bmv) \widetilde{\bm T}_{\bm j}(\bmvs)
   \widehat{\bm T}_\bmk(\bmv') 
   \, {\rm d} {\bm \sg} \, {\rm d} \bmv\, {\rm d} \bmvs,
\eea
\bea
   \widetilde{I}_2(\bm{i,j,k})
   &:=& 
   \int_{\R^d} \int_{\R^d} \int_{S^{d-1}}
  \mathcal{B} (\bmv,\bmvs,{\bm \sg})\,
   \widetilde{\bm T}_{\bm i}(\bmv) \widetilde{\bm T}_{\bm j}(\bmvs)
   \widehat{\bm T}_\bmk(\bmv) 
   \, {\rm d} {\bm \sg} \, {\rm d} \bmv\, {\rm d} \bmvs.
  \eea
Since the tensors $ \widetilde{I}_1(\bm{i,j,k})$ and $\widetilde{I}_2(\bm{i,j,k})$ do not depend on coefficients $\{ \wtf_{\bm k} \}_{0\leq {\bm k} \leq N}$, a straightforward way to evaluate $\mQ_{\bm{k}}^N $ is to precompute $ \widetilde{I}_1(\bm{i,j,k})$ and $\widetilde{I}_2(\bm{i,j,k})$, and then evaluate the sum in (\ref{direct_online}) directly in the online computation. This is what we refer to as the direct algorithm. We observe that this algorithm requires $\mO(N^{3d})$ memory to store the tensors $\widetilde{I}_1$ and $\widetilde{I}_2$; and to evaluate (\ref{direct_online}), it requires $\mO(N^{3d})$ operations. Both the memory requirement and online computational cost can be quite demanding especially for $d=3$ and large $N$.

We give some details on how to approximate $ \widetilde{I}_1(\bm{i,j,k})$ and $\widetilde{I}_2(\bm{i,j,k})$, though this step can be completed in advance and does not take the actual computational time. We first perform a change of variables $(\bmv,\bmvs)\rightarrow (\bmv(\bmxi),\bmvs(\bmeta))$ to transform the integrals of $(\bm{v,v_*}) \in \R^d \times \R^d$ into integrals of $(\bmxi, \bmeta)\in I^d \times I^d$, using the mapping introduced in Section~\ref{Subsec:Basis_function}:
\bea
   \widetilde{I}_1(\bm{i,j,k})
     &=& 
   \int_{I^d} \int_{I^d} 
   G_\bmk(\bmxi, \bmeta)
   \f{{\bm T}_{\bm i}(\bmxi) {\bm T}_{\bm j}(\bmeta) }{\sqrt{\bm{c_i c_j}}} {\bm \og}({\bm \xi}) {\bm \og}(\bmeta)
   \, {\rm d} {\bmxi}\, {\rm d} {\bmeta},
   \label{tensor_I1}
\eea
\bea
   \widetilde{I}_2(\bm{i,j,k})
     &=& 
   \int_{I^d} \int_{I^d} 
   L_\bmk(\bmxi, \bmeta)
   \f{{\bm T}_{\bm i}({\bmxi}) {\bm T}_{\bm j}(\bmeta) }{\sqrt{\bm{c_i c_j}}}   {\bm \og}({\bmxi}) {\bm \og}({\bmeta})
   \, {\rm d} {\bmxi}\, {\rm d} {\bmeta}, 
      \label{tensor_I2}
\eea
where
\bea 
   G_\bmk(\bmxi, \bmeta)
   &:=& 
   \left[{\bm \mu}({\bmxi}) {\bm \mu}({\bmeta})\right]^2
   \int_{S^{d-1}}\mathcal{B} (\bmv(\bmxi),\bmvs(\bmeta),{\bm \sg})
   \f{
   {\bm T}_\bmk \left( 
      \bmzeta_{(\bmxi, \bmeta, \bmsg)} 
    \right)}
   {\sqrt{\bm c_k} 
   [{\bm \mu}\left( 
      \bmzeta_{(\bmxi, \bmeta, \bmsg)} 
    \right)]^2}
   \, {\rm d} {\bm \sg} ,  \label{GGLL1} \\
   L_\bmk(\bmxi, \bmeta)
   &:=& 
   \f{{\bm T}_\bmk({\bmxi}) \left[{\bm \mu}({\bmeta}) \right]^2 }{\sqrt{\bm{c_k}}}
   \int_{S^{d-1}}
   \mathcal{B} (\bmv(\bmxi),\bmvs(\bmeta),{\bm \sg})
   \, {\rm d} {\bm \sg}.   \label{GGLL2}
\eea
Notice that in (\ref{GGLL1}), $\bmzeta_{(\bmxi, \bmeta, \bmsg)}\in I^d$ is the value transformed from 
$$\bmv' = \f{1}{2} (\bmv(\bmxi)+\bmvs(\bmeta)) + \f{1}{2} |\bmv(\bmxi)-\bmvs(\bmeta)| {\bm \sg}\in \mathbb{R}^d$$ under the same mapping. To approximate the above integrals in $\bmxi$, $\bmeta$, and $\bmsg$, we choose $M_\bmv$ Chebyshev-Gauss-Lobatto quadrature points in each dimension of $I^d$ for both $\bmxi$ and $\bmeta$; and $M_{\bmsg}$ quadrature points on the unit sphere $S^{d-1}$ (for $d=2$, this can be the uniform points in polar angle; for $d=3$, this can be the Lebedev quadrature \cite{lebedev1976quadratures}). Therefore, for each fixed index $\bmk$, (\ref{tensor_I1}) and (\ref{tensor_I2}) are forward Chebyshev transforms of the functions $G_\bmk(\bmxi, \bmeta)$ and $L_\bmk(\bmxi, \bmeta)$, respectively. Thus, they can be evaluated efficiently using the fast Chebyshev transform. 

\subsection{A fast algorithm}

To introduce the fast algorithm, we take the original form (\ref{Qk_component}) and split $\mQ_{\bm{k}}^{N}$ into a gain term and a loss term as $\mQ_{\bm{k}}^N =\mQ_{\bm{k}}^{N,+} -\mQ_{\bm{k}}^{N,-}$, where
\bea
   \mQ_{\bm{k}}^{N,+}  &=& \int_{\R^d} \left(\int_{\R^d} \int_{S^{d-1}} \mathcal{B} (\bmv,\bmvs,{\bm \sg}) f_N(\bmvs)\,  \widehat{\bm T}_\bmk(\bmv') \, {\rm d} {\bm \sg} \,{\rm d} \bmvs\right) f_N(\bmv)\, {\rm d} \bmv, \label{Qplus}\\
    \mQ_{\bm{k}}^{N,-}& =& \int_{\R^d} \left(\int_{\R^d} \int_{S^{d-1}} \mathcal{B} (\bmv,\bmvs,{\bm \sg}) f_N(\bmvs) \, {\rm d} {\bm \sg} \,{\rm d} \bmvs\right) f_N(\bmv)\,\widehat{\bm T}_\bmk(\bmv) \, {\rm d} \bmv. \label{Qminus}
\eea

We propose to evaluate $\mQ_{\bm{k}}^{N,+}$ and $\mQ_{\bm{k}}^{N,-}$ following the above expressions. To this end, given the coefficients $\{ \wtf_{\bm{k}} \}_{0 \le \bmk \le N}$ at each time step, we first reconstruct $f_N$ as in (\ref{fN_Expansion}) at $M_\bmv$ Chebyshev-Gauss-Lobatto quadrature points in each dimension of $\bmv$ (for an accurate approximation we choose $M_\bmv=N+2$). This can be achieved by the fast Chebyshev transform in $\mathcal{O}(M_{\bmv}^d\log M_{\bmv})$ operations.

{\bf To evaluate the gain term $\mQ_{\bm{k}}^{N,+}$}, we change the integrals of $(\bm{v,v_*}) \in \R^d \times \R^d$ into integrals of $(\bmxi, \bmeta)\in I^d \times I^d$ in (\ref{Qplus}) (similarly as in the previous subsection for $\widetilde{I}_1(\bm{i,j,k})$):
\begin{equation}
\begin{split}
 \mQ_{\bm{k}}^{N,+} 
         &=  \int_{I^d} 
   \left( \int_{I^d} 
      \int_{S^{d-1}} 
   \mathcal{B} (\bmv(\bmxi),\bmvs(\bmeta),{\bm \sg}) f_N(\bmvs(\bmeta))
   \f{{\bm T}_\bmk\left( 
      \bmzeta_{(\bmxi, \bmeta, \bmsg)} 
   \right)}{\sqrt{\bm{c}_\bmk}
   \left[\bmmu \left( 
      \bmzeta_{(\bmxi, \bmeta, \bmsg)} 
   \right)\right]^2}
   \f{\bm{\og(\eta)}}{\left[{\bm \mu}(\bm{\eta})  \right]^2}
      \, {\rm d} {\bm \sg}
       \, {\rm d} {\bm \eta}
   \right)
   f_N(\bmv(\bmxi)) \f{\bm{\og}(\bmxi)}{[\bm{\mu}(\bmxi)]^2} \, {\rm d} {\bm \xi}\\
 &=  \int_{I^d} \left(\int_{S^{d-1}}F_\bmk ({\bmsg},\bmxi)\, {\rm d} {\bm \sg}\right) f_N(\bmv(\bmxi)) \f{\bm{\og}(\bmxi)}{[\bm{\mu}(\bmxi)]^2} \, {\rm d} {\bm \xi}, \label{Qplus1}
\end{split}
\end{equation}
where
\begin{equation} \label{FF}
F_\bmk ({\bmsg},\bmxi):=  \int_{I^d} 
  \mathcal{B} (\bmv(\bmxi),\bmvs(\bmeta),{\bm \sg}) f_N(\bmvs(\bmeta))
   \f{{\bm T}_\bmk\left( 
      \bmzeta_{(\bmxi, \bmeta, \bmsg)} 
   \right)}{\sqrt{\bm{c}_\bmk}
   \left[\bmmu \left( 
      \bmzeta_{(\bmxi, \bmeta, \bmsg)} 
   \right)\right]^2}
   \f{\bm{\og(\eta)}}{\left[{\bm \mu}(\bm{\eta})  \right]^2}
       \, {\rm d} {\bm \eta}.
\end{equation}
Suppose $M_\bmv$ quadrature points are used in each dimension of ${\bmv}$ and ${\bmvs}$ and $M_{\bmsg}$ points are used on the sphere $S^{d-1}$, a direct evaluation of (\ref{FF}) would require $\mathcal{O}(M_{\bmsg}M_{\bmv}^{2d}N^d)$ operations. Given $F_\bmk ({\bmsg},\bmxi)$, a direct evaluation of (\ref{Qplus1}) would take $\mathcal{O}(M_{\bmsg}M_{\bmv}^{d}N^d)$ operations. Therefore, the major bottleneck is to compute $F_\bmk ({\bmsg},\bmxi)$, which is prohibitively expensive without a fast algorithm. Our main idea is to recognize (\ref{FF}) as a non-uniform discrete Fourier cosine transform so it can be evaluated by the non-uniform fast Fourier transform  (NUFFT). We will see that the total complexity to evaluate $F_\bmk ({\bmsg},\bmxi)$ can be brought down to $\mathcal{O}(M_{\bmsg}M_{\bmv}^{2d} |\log \epsilon|+M_{\bmsg}M_{\bmv}^dN^d\log N)$, where $\epsilon$ is the requested precision in the NUFFT algorithm.

Applying the Chebyshev-Gauss-Lobatto quadrature $(\bmeta_{\bmj},w_{\bmj})_{1\leq \bmj \leq M_{\bmv}}$, (\ref{FF}) becomes
\bea
\label{num_fast_k}
F_\bmk(\bmsg, \bmxi) &\approx&
\sum_{1\leq \bmj \leq M_{\bmv}} w_\bmj
\f{ \mB (\bmv(\bmxi), \bmvs(\bmeta_\bmj), \bmsg) f_N(\bmvs(\bmeta_\bmj))}{\sqrt{c_\bmk} [\bmmu(\bmzeta_{(\bmxi, \bmeta_\bmj, \bmsg)})]^2[\bmmu(\bmeta_\bmj)]^2}
{\bm T}_\bmk(\bmzeta_{(\bmxi, \bmeta_\bmj, \bmsg)}) 
\nn \\
& = &
\sum_{1\leq \bmj \leq M_{\bmv}} w_\bmj
\f{ \mB (\bmv(\bmxi), \bmvs(\bmeta_\bmj), \bmsg) f_N(\bmvs(\bmeta_\bmj))}{\sqrt{c_\bmk} [\bmmu(\bmzeta_{(\bmxi, \bmeta_\bmj, \bmsg)})]^2[\bmmu(\bmeta_\bmj)]^2}
\prod_{l=1}^d \cos\left(k_l \arccos( \bmzeta_{(\bmxi, \bmeta_\bmj, \bmsg),l})\right)
\nn \\
& = & 
\f{1}{\sqrt{c_\bmk}} 
\sum_{1\leq \bmj \leq M_{\bmv}}
q_\bmj  
\prod_{l=1}^d \cos\left(k_l \bmz_{\bmj, l} \right), 
\eea 
where $\bmzeta_{(\bmxi, \bmeta_\bmj, \bmsg),l}$ is the $l$-th component of $\bmzeta_{(\bmxi, \bmeta_\bmj, \bmsg)}$, and 
\begin{equation}
 q_\bmj := w_\bmj
\f{ \mB (\bmv(\bmxi), \bmvs(\bmeta_\bmj), \bmsg) f_N(\bmvs(\bmeta_\bmj))}{[\bmmu(\bmzeta_{(\bmxi, \bmeta_\bmj, \bmsg)})]^2[\bmmu(\bmeta_\bmj)]^2}, \quad \bmz_{\bmj, l}:=\arccos( \bmzeta_{(\bmxi, \bmeta_\bmj, \bmsg),l}).
\end{equation}
Note that $q_\bmj$ and $\bmz_\bmj$ depend on $\bmsg$ and $\bmxi$. (\ref{num_fast_k}) is almost like a non-uniform discrete Fourier cosine transform. Indeed, we propose to evaluate $F_\bmk(\bmsg, \bmxi)$ as follows. 

For each fixed $\bmsg$ and $\bmxi$, we compute
\be \label{nufft}
\widetilde{F}_\bmK:= \sum_{1\leq \bmj \leq M_{\bmv}}q_\bmj\exp(\bfi \bmK \cdot \bmz_{\bmj}), \quad -N\leq \bmK\leq N,
\ee
which is a non-uniform discrete Fourier transform mapping non-uniform samples $\bmz_{\bmj}\in [0,\pi]^d$ into frequencies $\bmK \in[-N,N]^d$. This can be done efficiently using the NUFFT. In recent years, various NUFFT algorithms have been developed. In our numerical realization, we employ an efficient algorithm called FINUFFT \cite{barnett2019parallel}. The general idea is to apply an interpolation between non-uniform samples and an equispaced grid, and then perform the uniform FFT on the new grid. This algorithm only costs $\mO(M_{\bmv}^d|\log \ep|+N^d \log N)$ operations in computing (\ref{nufft}) with the requested precision $\ep$. Once we obtain $\widetilde{F}_\bmK$, $F_\bmk(\bmsg, \bmxi)$ can be retrieved as
\begin{itemize}
\item in 2D
\be
F_\bmk(\bmsg, \bmxi)=\frac{1}{\sqrt{c_\bmk}} \frac{1}{2}  {\bf Re} \left(\widetilde{F}_{(k_1,k_2)}+\widetilde{F}_{(-k_1,k_2)}\right);
\ee
\item in 3D
\be
F_\bmk(\bmsg, \bmxi)=\frac{1}{\sqrt{c_\bmk}} \frac{1}{4}  {\bf Re} \left(\widetilde{F}_{(k_1,k_2,k_3)}+\widetilde{F}_{(-k_1,k_2,k_3)}+\widetilde{F}_{(k_1,-k_2,k_3)}+\widetilde{F}_{(k_1,k_2,-k_3)}\right).
\ee
\end{itemize}
This procedure needs to be repeated for every $\bmsg$ and $\bmxi$, hence the overall computational cost for getting $F_\bmk(\bmsg, \bmxi)$ is $\mathcal{O}(M_{\bmsg}M_{\bmv}^{2d} |\log \epsilon|+M_{\bmsg}M_{\bmv}^dN^d\log N)$.

{\bf To evaluate the loss term $\mQ_{\bm{k}}^{N,-}$}, we change the integrals of $(\bm{v,v_*}) \in \R^d \times \R^d$ into integrals of $(\bmxi, \bmeta)\in I^d \times I^d$ in (\ref{Qminus})  (similarly as in the previous subsection for $\widetilde{I}_2(\bm{i,j,k})$):
\begin{equation}
\mQ_{\bm{k}}^{N,-}=\int_{I^d} \left(\int_{I^d} \int_{S^{d-1}}
  \mathcal{B} (\bmv(\bmxi),\bmvs(\bmeta),{\bm \sg}) 
    f_N(\bmvs(\bmeta)) \f{\bm{\og(\eta)}}{\left[{\bm \mu}(\bm{\eta})  \right]^2} \,
    {\rm d} {\bm \sg}  
       \, {\rm d} {\bm \eta} \right) f_N(\bmv(\bmxi))\,  \f{{\bm T}_\bmk({\bmxi}) \bm{\og(\bmxi)}  }{\sqrt{\bm{c_k}} \left[{\bm \mu}({\bmxi}) \right]^4}
 \, {\rm d} \bmxi.
\end{equation}
Then one can just evaluate the terms in the parentheses directly with complexity $\mathcal{O}(M_{\bmv}^{2d}M_{\bmsg})$. The outer integral in $\bmxi$ can be viewed as the Chebyshev transform of some function thus can be evaluated efficiently by the fast Chebyshev transform in $\mathcal{O}(M_{\bmv}^d\log M_{\bmv})$. In particular, if we consider the Maxwell kernel, i.e., $\mathcal{B} (\bmv,\bmvs,{\bm \sg})\equiv \text{constant}$, terms in the parentheses only requires $\mathcal{O}(M_{\bmv}^{d})$ complexity. 

\subsection{Comparison of direct and fast algorithms}

To summarize, we list the storage requirement and (online) computational complexity for both the direct and fast algorithms in Table 1. Note that we only list the dominant complexity for each term. It is clear that the main cost of the fast algorithm comes from evaluating the gain term. Compared with the direct algorithm, the fast algorithm is generally faster as $M_{\bmsg}$ can be chosen much smaller than $N^d$ in practice (see Section~\ref{Sec:num}). Most importantly, the fast algorithm does not require any precomputation with excessive storage requirement and everything can be computed on the fly.

\begin{table}[H]  
   \begin{center}
   \begin{tabular}{ c|c|c|c} 
   \hline
   & \multicolumn{2}{c|}{direct algorithm} &  fast algorithm \\ 
   \hline
   & storage  & (online) operation & (online) operation \\
   \hline
   gain term & $\mO(N^{3d})$ & $ \mO(N^{3d})$  & $\mathcal{O}(M_{\bmsg}M_{\bmv}^{2d} |\log \epsilon|+M_{\bmsg}M_{\bmv}^dN^d\log N)$ \\
   loss term &  $\mO(N^{3d})$ & $ \mO(N^{3d})$ & $\mathcal{O}(M_{\bmsg}M_{\bmv}^{2d})$\\
   \hline 
    \end{tabular}
   \caption{Storage requirement and (online) computational cost of the direct and fast algorithms. $N$ is the number of spectral modes in each dimension of $\bmv$; $M_{\bmv}=\mO(N)$ is the number of quadrature points in each dimension; $M_{\bmsg}\ll N^d$ is the number of quadrature points on the sphere $S^{d-1}$; and $\ep$ is the requested precision in the NUFFT algorithm. The proposed fast algorithm does not require extra storage other than that storing the computational target, e.g., the gain and loss terms. }
   \end{center}
\end{table}

\section{Numerical examples}
\label{Sec:num}

In this section, we perform extensive numerical tests to demonstrate the accuracy and efficiency of the proposed Petrov-Galerkin spectral method in both 2D and 3D.

Recall that the main motivation of the current work is to obtain better accuracy by considering approximations in an unbounded domain. To illustrate this point, we will compare three methods to solve the Boltzmann equation:
\begin{description}
   \item[(1) Fast Fourier-Galerkin spectral method proposed in \cite{gamba2017fast}:] this method can achieve a good accuracy-efficiency tradeoff among the current deterministic methods for the Boltzmann equation. However, it requires the truncation of the domain to $[-L,L]^d$, where $L$ is often chosen empirically such that the solution is close to zero at the boundary.
   \item[(2) Fast Chebyshev-0 method:] the method proposed in the current paper using the logarithmic mapping (\ref{logarithmic}), where $r=0$ and the scaling parameter $S$ needs to be properly chosen.
   \item[(3) Fast Chebyshev-1 method:] the method proposed in the current paper using the algebraic mapping (\ref{algebraic}), where $r=1$ and the scaling parameter $S$ needs to be properly chosen.
\end{description}

In all three methods, the choice of truncation or mapping/scaling parameters has a great impact on the numerical accuracy. In the following tests, we first determine $L$ in the Fourier spectral method. Then for the two Chebyshev methods, we propose an adaptive strategy to determine the scaling parameter $S$: for example, in 1D, the Chebyshev-Gauss-Lobatto quadrature points on the interval $[-1, 1]$ are given by
$$
-1 = \xi_1 < \xi_2 < \ldots < \xi_{N} = 1.
$$
We choose $S$ such that the two quadrature points $\xi_2$ and $\xi_{N-1}$ are mapped to the boundary of $[-L, L]$, i.e.,
$$
v(\xi_1)=-\infty, \quad v(\xi_2) = -L, \quad v(\xi_{N-1}) = L, \quad v(\xi_N)=-\infty.
$$
Note that this $S$ is adaptive in the sense that different $N$ will correspond to different $S$.


\subsection{2D examples}

\subsubsection{2D BKW solution}

We consider first the 2D BKW solution. This is one of the few known analytical solutions to the Boltzmann equation and a perfect example to verify the accuracy of a numerical method. 

When $d=2$ and the collision kernel $\mathcal{B}\equiv 1/(2\pi)$, the following is a solution to the initial value problem (\ref{MultiDBoltz}):
\be
   f_{\rm BKW}(t,\bmv) = 
   \f{1}{2\pi K^2} \exp\left(-\f{\bmv^2}{2K}\right)
   \left(2K-1 + \f{1-K}{2K} \bmv^2\right),
   \label{2d_bkw_initial}
\ee
where $K=1-\exp(-t/8)/2$. By taking the time derivative of $ f_{\rm BKW}(t,\bmv)$, we can obtain the exact collision operator as
\be
   Q_{\rm BKW}(f)
   = \left\{ \left(-\f{2}{K} + \f{\bmv^2}{2K^2}\right)f_{\rm BKW} + 
   \f{1}{2\pi K^2} \exp\left(-\f{\bmv^2}{2K}\right)
   \left(2 - \f{1}{2K^2} \bmv^2\right)
   \right\} K',
   \label{2d_bkw_rhs}
\ee
where $K' = \exp(-t/8)/16$. This way we can apply the numerical method to compute $ Q_{\rm BKW}(f)$ directly and check its error without worrying about the time discretization.

In the fast Fourier spectral method, we take $N_\rho = N$ quadrature points in the radial direction and $N_\tht = 8$ quadrature points on the unit circle (see \cite{gamba2017fast} for more details). In the fast Chebyshev methods, we take $M_\bmv = N+2$ quadrature points for each dimension of $(\bmv, \bmvs)$ and $M_{\bmsg} = N$ quadrature points on the unit circle. The precision in NUFFT is selected as $\ep = 1e-14$. The numerical error of $Q_{\rm BKW}(f)$ is estimated on a $200 \times 200$ uniform grid in the rectangular domain $ [-6.3, 6.3]^2$ at time $t=2$. 

{\bf Test 01:} In this test, we examine thoroughly the numerical errors concerning different truncation parameters $L$ in the Fourier method, and scaling parameters $S$ in the Chebyshev methods. The $L^2$ errors of $Q_{\rm BKW}(f)$ for three methods are presented in Figure~\ref{2d_bkw_01}. It is obvious that the accuracy is not good when $L$ and $S$ are too small or too large. When $L$ and $S$ are chosen appropriately, the accuracy is close to the machine precision. In Table~\ref{2d_bkw_02}, we record the best accuracy for a given $N$ of each method. One can see that the fast Chebyshev-0 method can always achieve the best accuracy.

\begin{figure}[htp!]
   \centering
   \includegraphics[width=.59\linewidth]{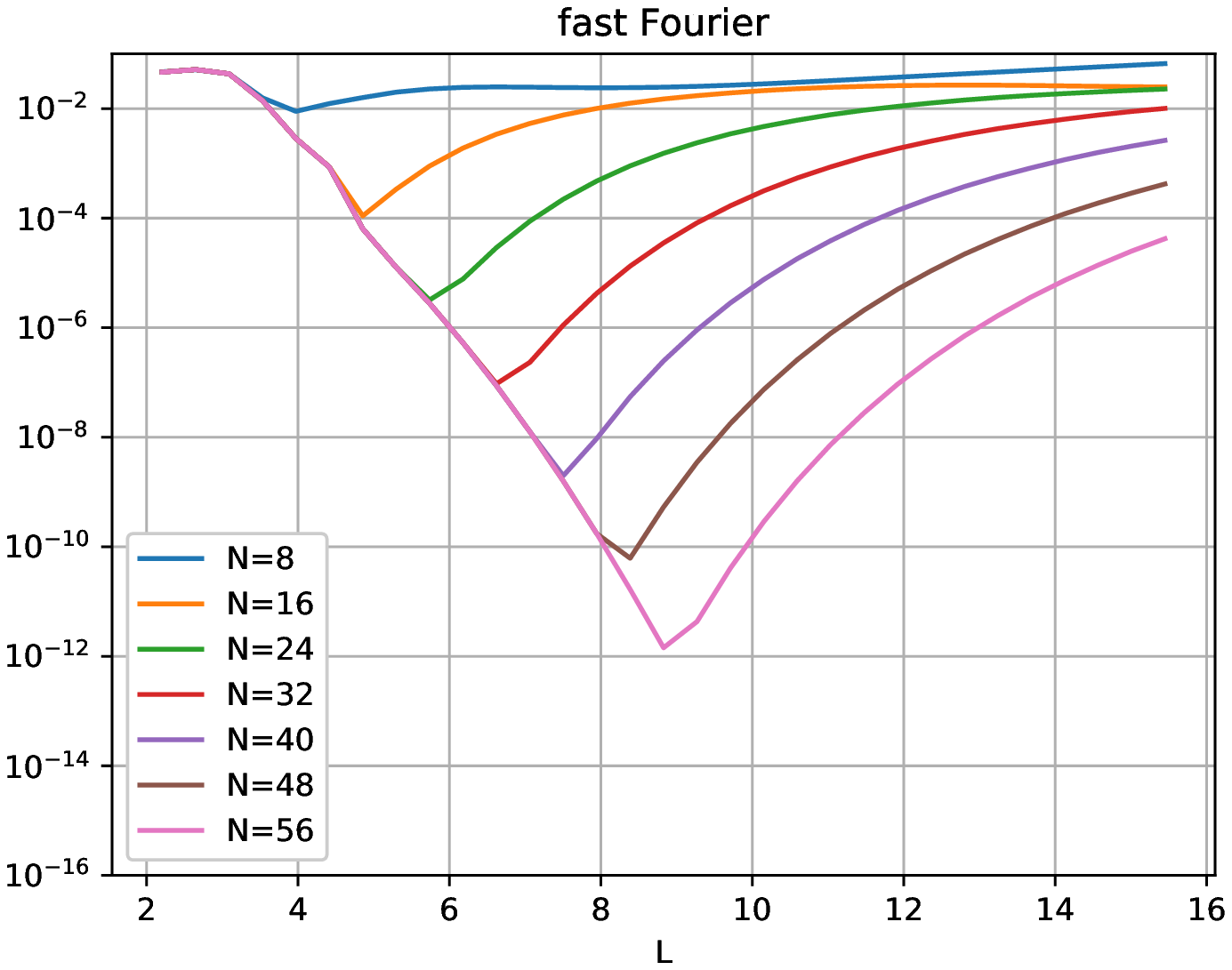}
   \includegraphics[width=.49\linewidth]{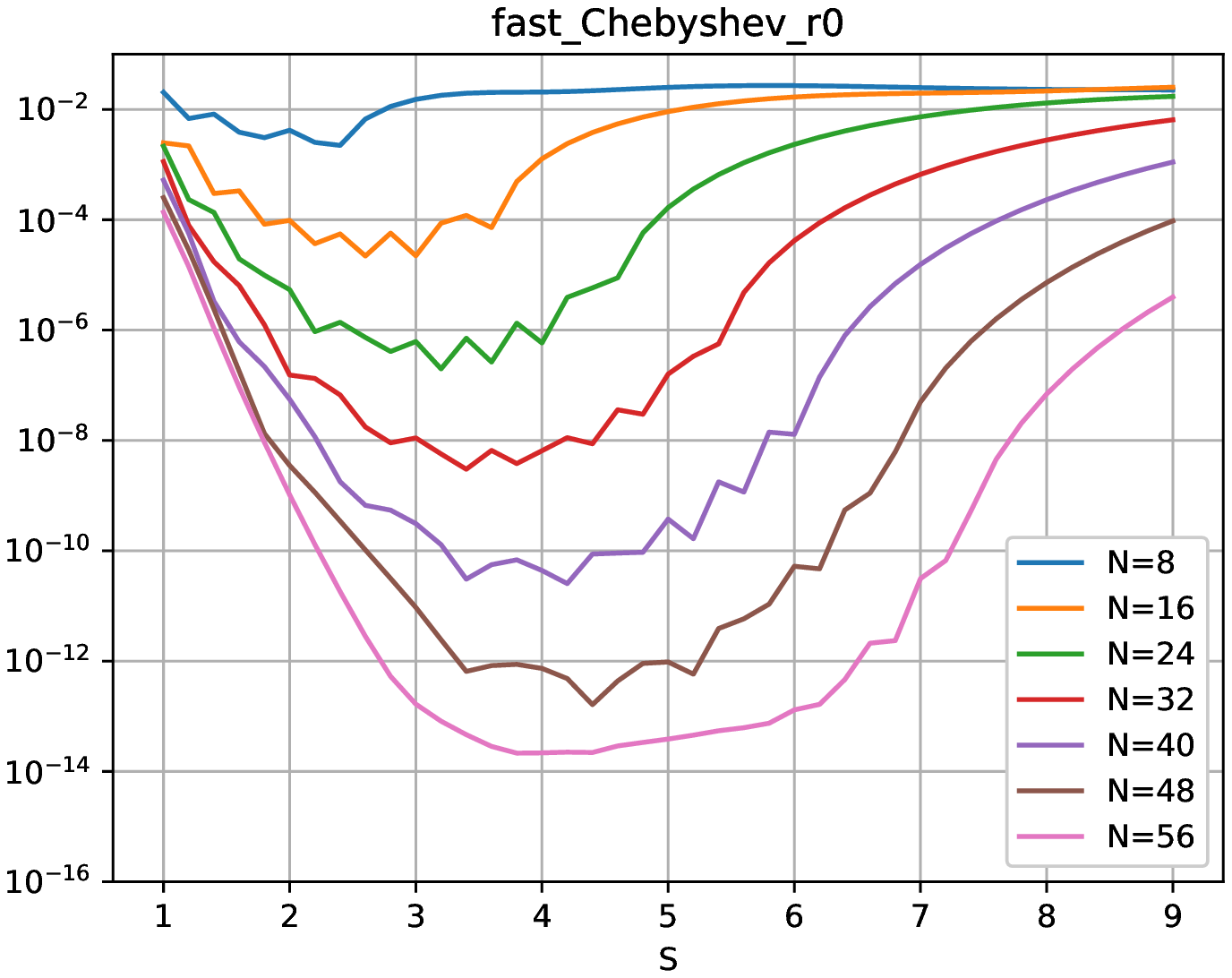}
   \includegraphics[width=.49\linewidth]{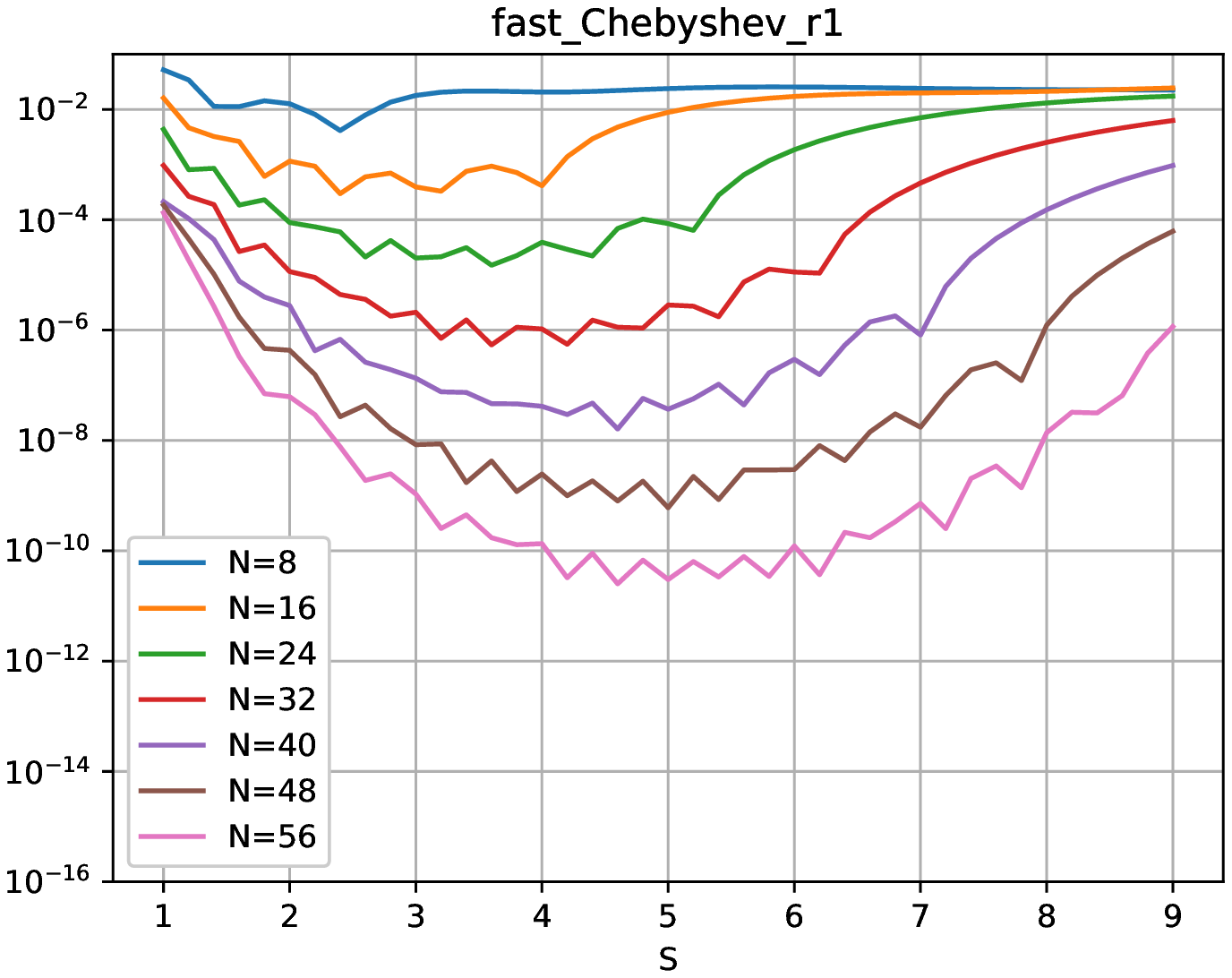}
   \caption{(2D BKW: Test 01) The $L^2$ error of $Q_{\rm BKW}(f)$ at time $t=2$. Top: fast Fourier method. 
   Bottom: fast Chebyshev methods.}  
   \label{2d_bkw_01}
\end{figure}

\begin{table}[htp!]
   \centering
   \begin{tabular}{l|c|c|c}
   \hline
   & fast Fourier & fast Chebyshev-0 & fast Chebyshev-1 \\ \cline{1-1}
   \hline
   $N = 8$  & 8.9223e-03 & 2.2289e-03 & 4.1388e-03 \\ \hline
   $N = 16$ & 1.0989e-04 & 2.2033e-05 & 2.9713e-04 \\ \hline
   $N = 24$ & 3.2104e-06 & 1.9843e-07 & 1.5004e-05 \\ \hline
   $N = 32$ & 9.4720e-08 & 3.0082e-09 & 5.3946e-07 \\ \hline
   $N = 40$ & 1.9836e-09 & 2.5434e-11 & 1.6120e-08 \\ \hline
   $N = 48$ & 6.1797e-11 & 1.6255e-13 & 6.0320e-10 \\ \hline
   $N = 56$ & 1.4315e-12 & 2.1482e-14 & 2.5213e-11 \\ \hline
   \end{tabular}
   \caption{(2D BKW: Test 01) The $L^2$ error of $Q_{\rm BKW}(f)$ at time $t=2$. The best accuracy for a given $N$ of each method.}
   \label{2d_bkw_02}
\end{table}

{\bf Test 02:} In this test, we fix the computational domain and examine the numerical errors concerning different $N$. In the Fourier method, we test $L=8.83$ and $L=13.24$. In the Chebyshev methods, we use the same $L$ to select the scaling parameter $S$ accordingly. The $L^\ift$ errors of $Q_{\rm BKW}(f)$ for three methods are presented in Figure~\ref{2d_bkw_03}. Among these three methods, the fast Chebyshev-0 method can achieve the best accuracy when $N$ is small. The fast Chebyshev-1 method doesn't provide a good approximation. This is because the quadrature points in Chebyshev-1 method are much more clustered near the origin compared to Chebyshev-0 method and apparently points located far away from the origin play an important role in the unbounded domain problem.


\begin{figure}[htp!]
   \centering
   \includegraphics[width=.49\linewidth]{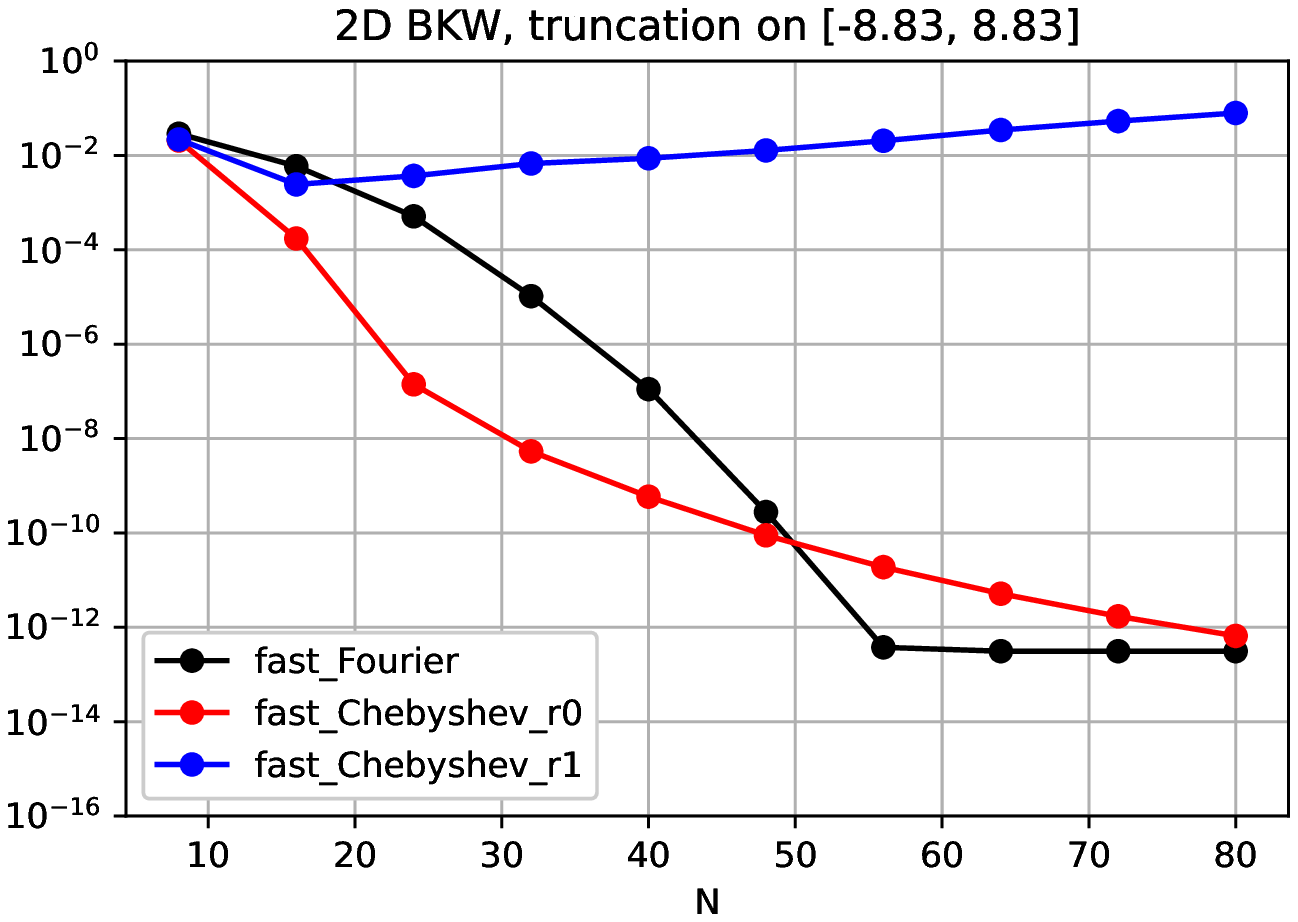}
   \includegraphics[width=.49\linewidth]{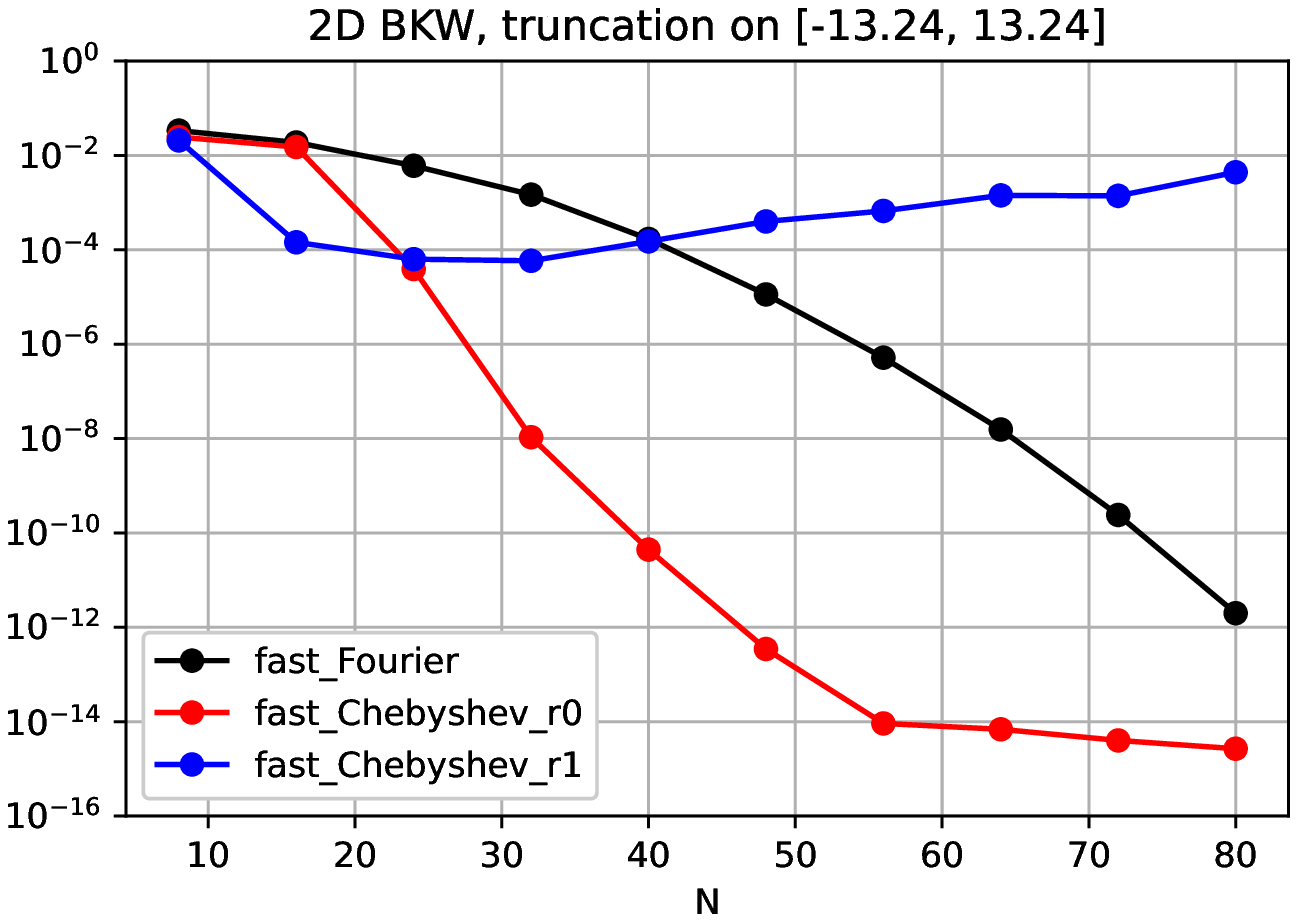}
   \caption{(2D BKW: Test 02) The $L^\infty$ error of $Q_{\rm BKW}(f)$ at time $t=2$. 
   Left: $L=8.83$;
   Right: $L=13.24$.}
   \label{2d_bkw_03}
\end{figure}

{\bf Test 03:} In this test, we examine the numerical errors of the Chebyshev methods with a fixed scaling parameter: $S=4$ in the fast Chebyshev-0 method; $S=5$ in the fast Chebyshev-1 method. These two values are selected based on the results in Figure~\ref{2d_bkw_01}. The $L^\ift$ errors of $Q_{\rm BKW}(f)$ for both methods are presented in Figure~\ref{2d_bkw_04}. As a comparison, results of the fast Fourier method are also plotted. Again for small $N$, the fast Chebyshev-0 method provides the best accuracy among three methods.

\begin{figure}[htp!]
   \centering
   \includegraphics[width=.59\linewidth]{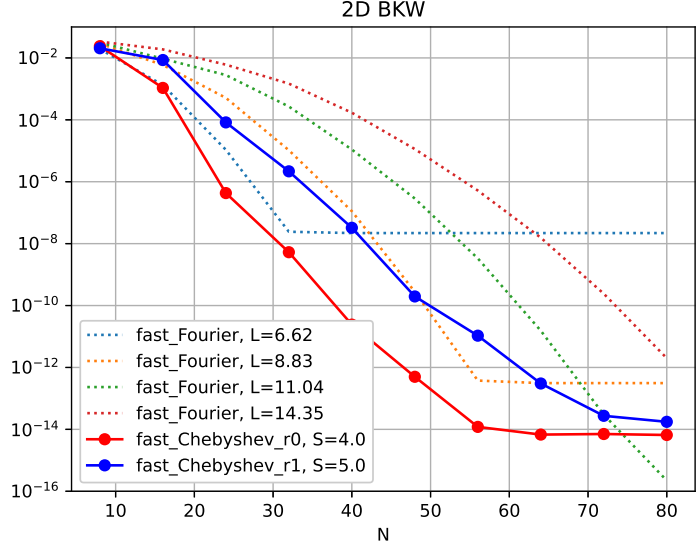}
   \caption{(2D BKW: Test 03) The $L^\infty$ error of $Q_{\rm BKW}(f)$ at time $t=2$. }
   \label{2d_bkw_04}
\end{figure}

{\bf Test 04:} In this test, we report the computational time of the direct (Chebyshev) algorithm and the fast (Chebyshev) algorithm. The computations were done on Intel(R) Core(TM) i7-6700 CPU in a single thread. Table~\ref{2d_BKW_time} shows the running time of the direct and fast algorithms concerning different $N$. Note that the direct algorithm is left out when $N\ge 32$ due to the memory constraint.


\begin{table}[htp!]
    \begin{center}
    \begin{tabular}{ c|c|c|c } 
    \hline
    & \multicolumn{2}{c|}{direct algorithm} & fast algorithm \\ 
    \hline
    & online {(sec)} & precomputation {(sec)} & online {(sec)} \\
    \hline
    $N=8$  
    & 0.0047 & 1.5956  & 0.194955 \\ \hline
    $N=16$ 
    & 0.1207 & 62.7991 & 2.036821 \\ \hline
    $N=32$ 
    & - & - & 24.779492 \\ \hline
    $N=64$  
    & - & - & 4.937722e+02\\ \hline
    $N=128$  
    & - & - & 1.331576e+04\\ \hline
    \end{tabular}
    \caption{(2D BKW: Test 04) Running time in second for a single evaluation of the gain term.}\label{2d_BKW_time}
    \end{center}
\end{table}


\subsubsection{Computing the moments}

We next consider the time evolution problem and check the accuracy for moments approximation. Since the fast Chebyshev-0 method performs generally better than the fast Chebyshev-1 method, we will restrict to the former in the following tests. The comparison with the fast Fourier method will still be considered.

In (\ref{MultiDBoltz}), we choose the collision kernel $\mathcal{B}\equiv 1/(2\pi)$ and the initial condition as
\be
   f^0(\bmv) = 
   \f{\rho_1}{2\pi T_1} \exp\left(-\f{(\bmv-V_1)^2}{2T_1}\right) + 
   \f{\rho_2}{2\pi T_2} \exp\left(-\f{(\bmv-V_2)^2}{2T_2}\right),
   \label{2d_maxwell_initial}
\ee
where $\rho_1 = \rho_2 = 1/2$, $T_1 = T_2 = 1$ and $V_1 = (x_1,y_1) = (-1,2)$, $V_2 = (x_2,y_2) = (3,-3)$.

Then for the momentum flow and energy flow defined as
\be
   P_{ij} = \int_{\R^2}f\bmv_i \bmv_j \, {\rm d} \bmv,\quad (i,j=1,2), \quad q_{i} = \int_{\R^2}f\bmv_i |\bmv|^2 \, {\rm d} \bmv,\quad (i=1,2),
\ee
we have their exact formulas 
\be
   P_{11} = -\f{9}{8} e^{-t/2} +\f{57}{8},\quad
   P_{12} = P_{21} =-5 e^{-t/2}-\f{1}{2},\quad
   P_{22} = \f{9}{8} e^{-t/2} +\f{51}{8},
\ee
and
\be
   q_1 = \f{1}{4} \left(11 e^{-t/2} +103 \right),\quad
   q_2= -\f{1}{8} \left(89 e^{-t/2} + 103 \right).
\ee

In the fast Fourier method, we take $N_\rho = N$ quadrature points in the radial direction and $N_\tht = N$ quadrature points on the unit circle. The truncation domain $[-L,L]^2$ is selected as $L=14.35$. In the fast Chebyshev-0 method, we take $M_\bmv = N+2$ quadrature points for each dimension of $(\bmv, \bmvs)$ and $M_{\bmsg} = N$ quadrature points on the unit circle. The precision in NUFFT is selected as $\ep = 1e-14$. The scaling parameter $S$ is adaptively chosen based on $L$. For both methods, we use the 4th-order Runge-Kutta method with $\Dt t = 0.02$ for time integration.

The absolute errors of the moments are presented in Figure~\ref{2d_P11}--\ref{2d_q2}. 
The fast Chebyshev-0 method clearly provides a better approximation in comparison to the Fourier method for fixed $N$.

\begin{figure}[htp!]
   \centering
   \includegraphics[width=.49\linewidth]{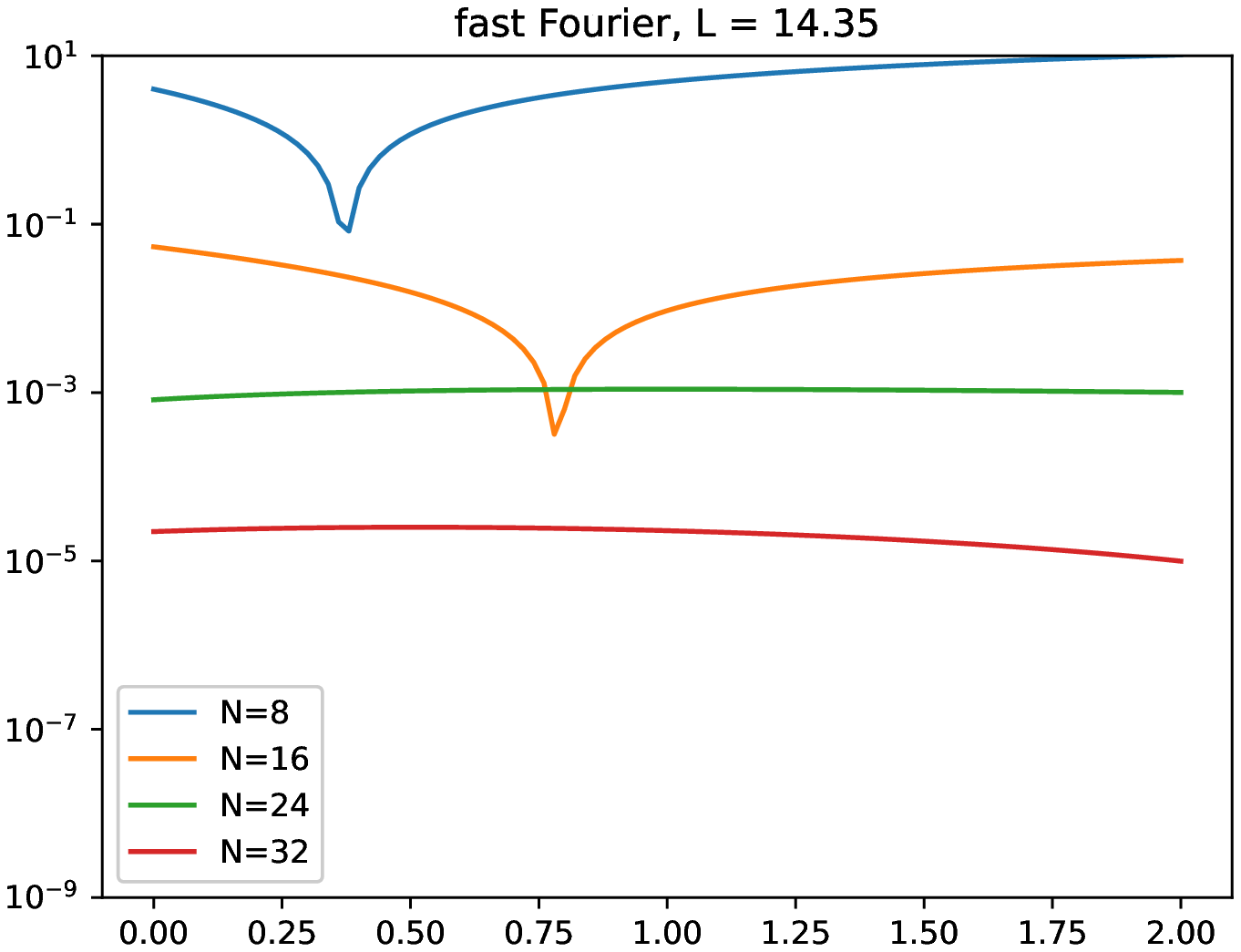}
   \includegraphics[width=.49\linewidth]{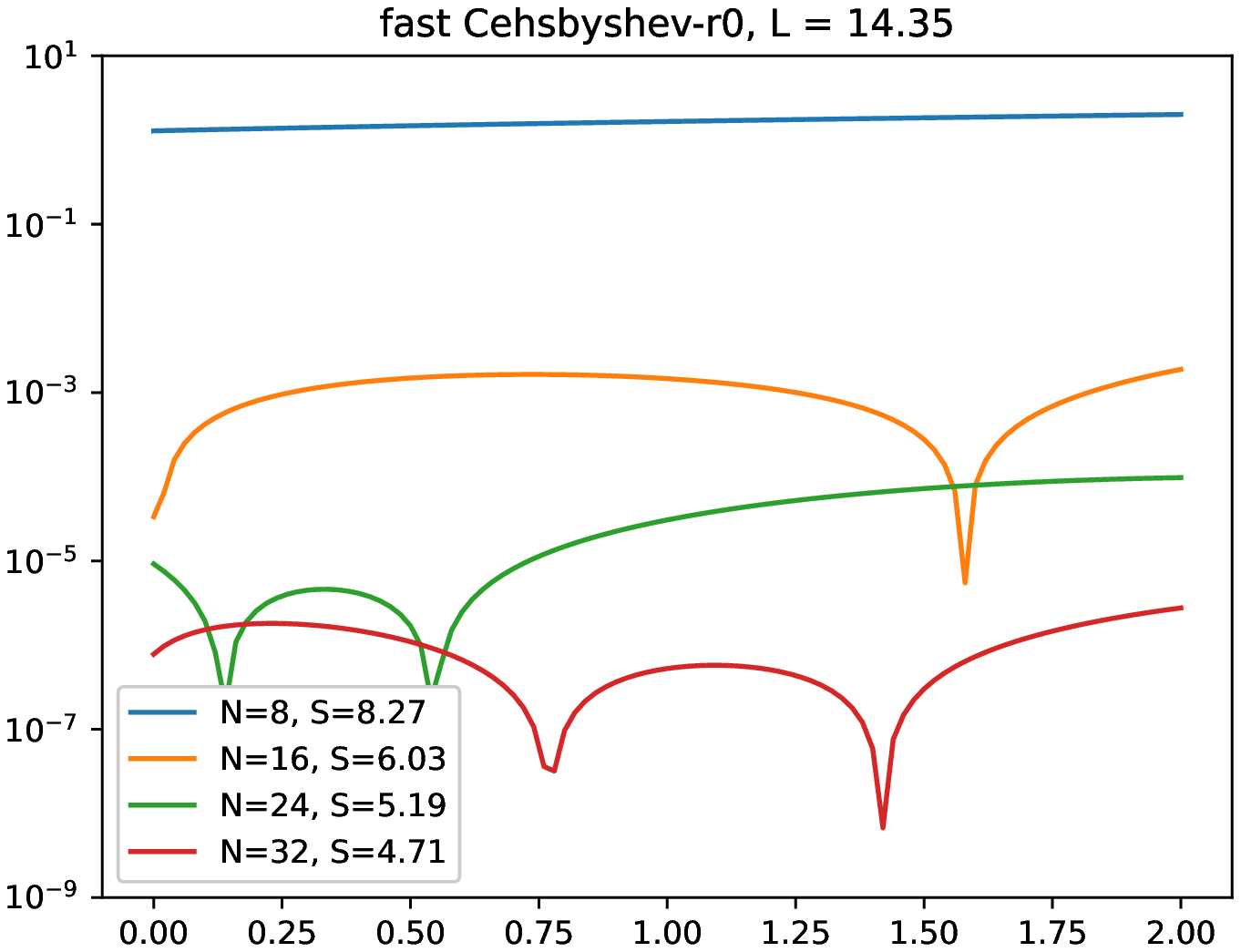}   
   \caption{(2D moments)
   The time evolution for the absolute error of the momentum flow $P_{11}$. 
   Left: the fast Fourier method.
   Right: the fast Chebyshev-0 method. }
   \label{2d_P11}
\end{figure}

\begin{figure}[htp!]
   \centering
    \includegraphics[width=.49\linewidth]{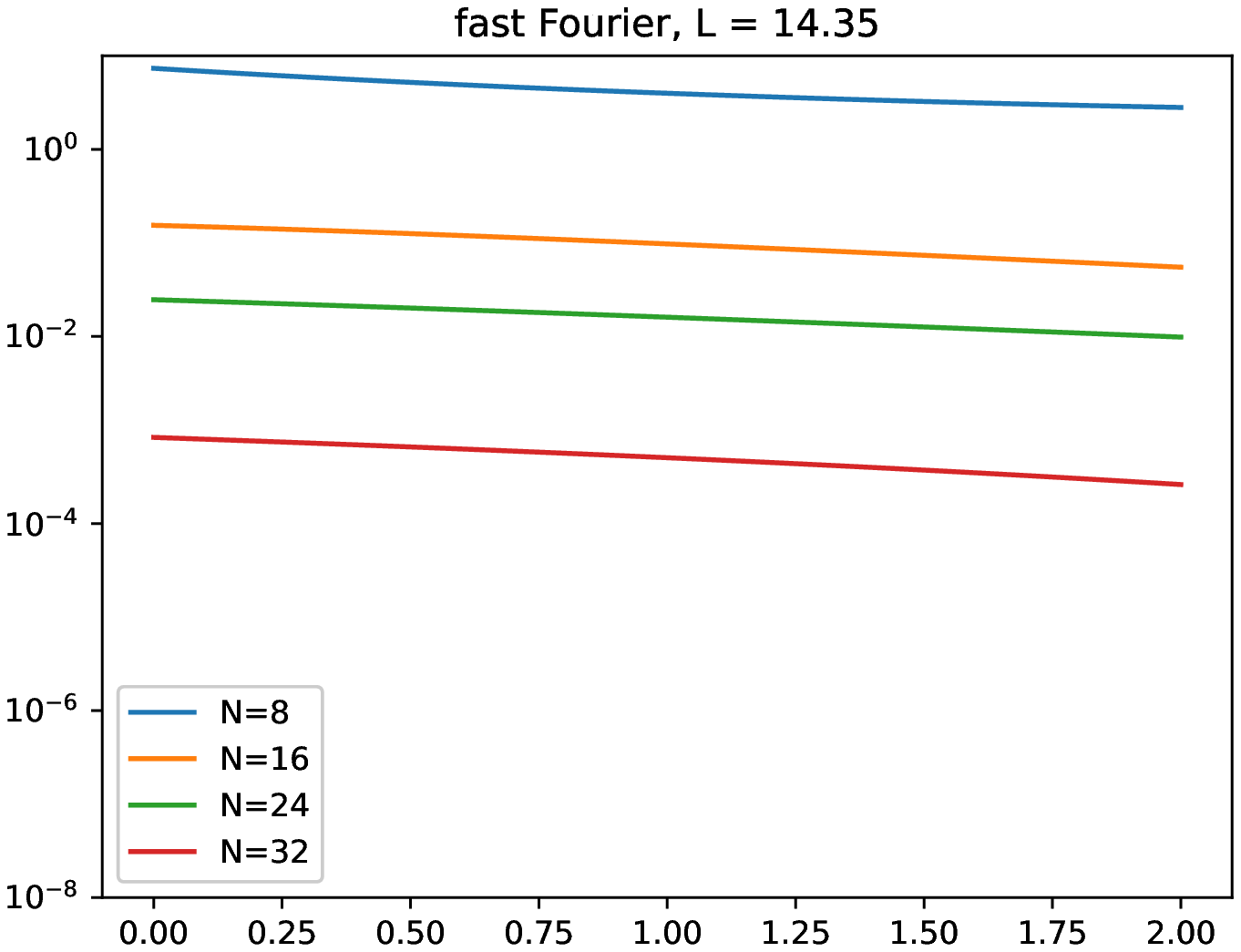}
   \includegraphics[width=.49\linewidth]{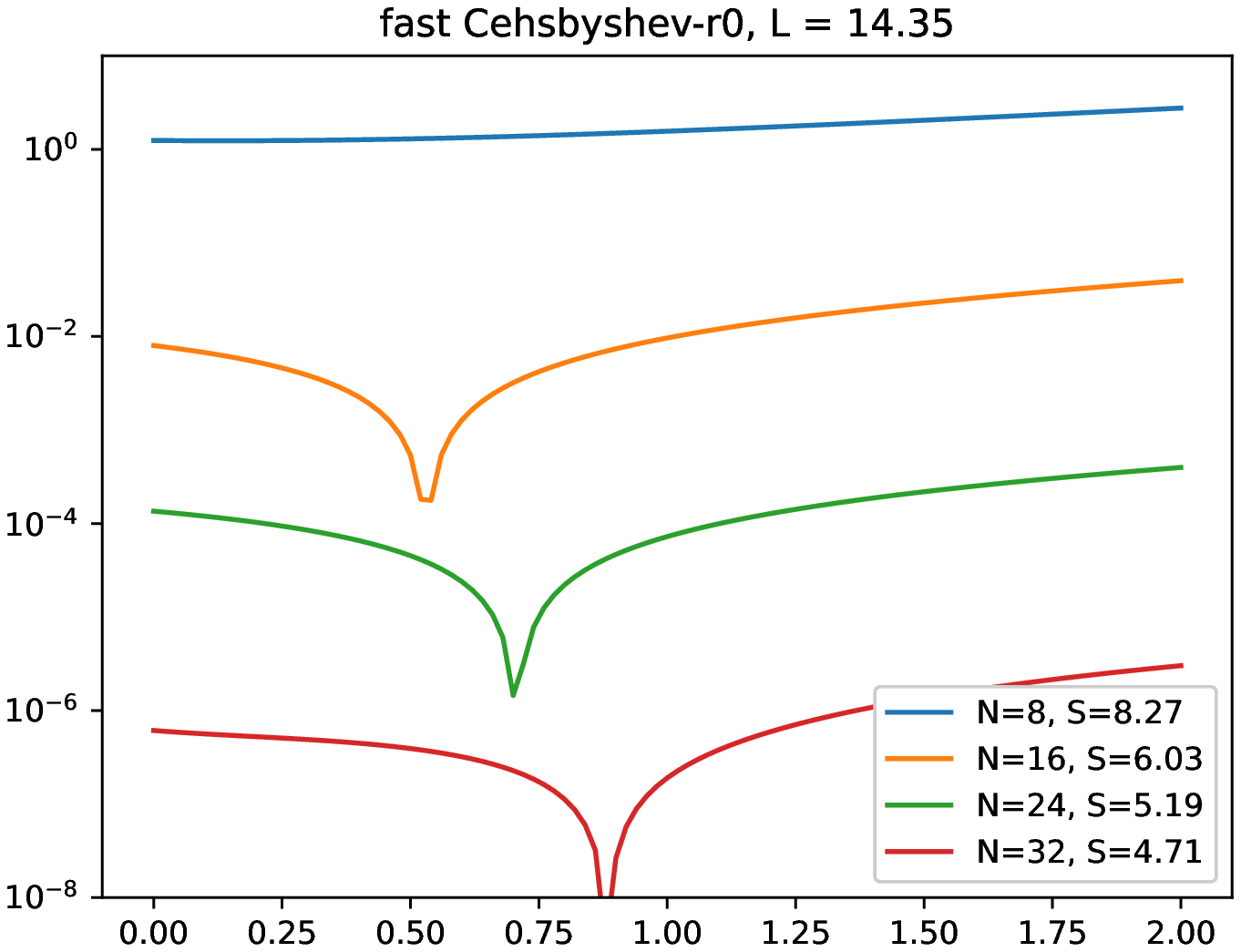}   
   \caption{(2D moments) 
   The time evolution for the absolute error of the momentum flow $P_{12}$.
   Left: the fast Fourier method.
   Right: the fast Chebyshev-0 method. }
   \label{2d_P12}
\end{figure}

\begin{figure}[htp!]
   \centering
   \includegraphics[width=.49\linewidth]{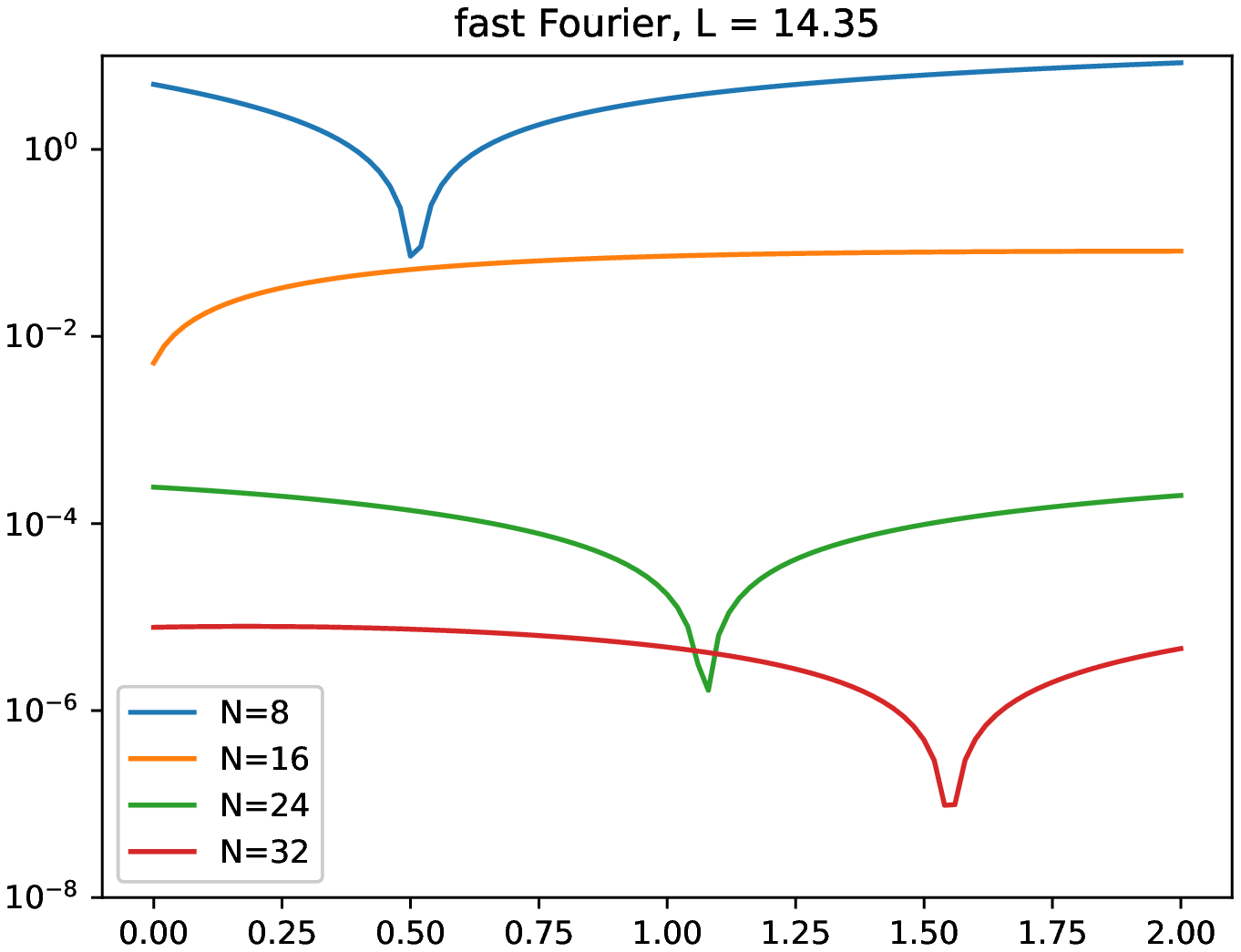}
   \includegraphics[width=.49\linewidth]{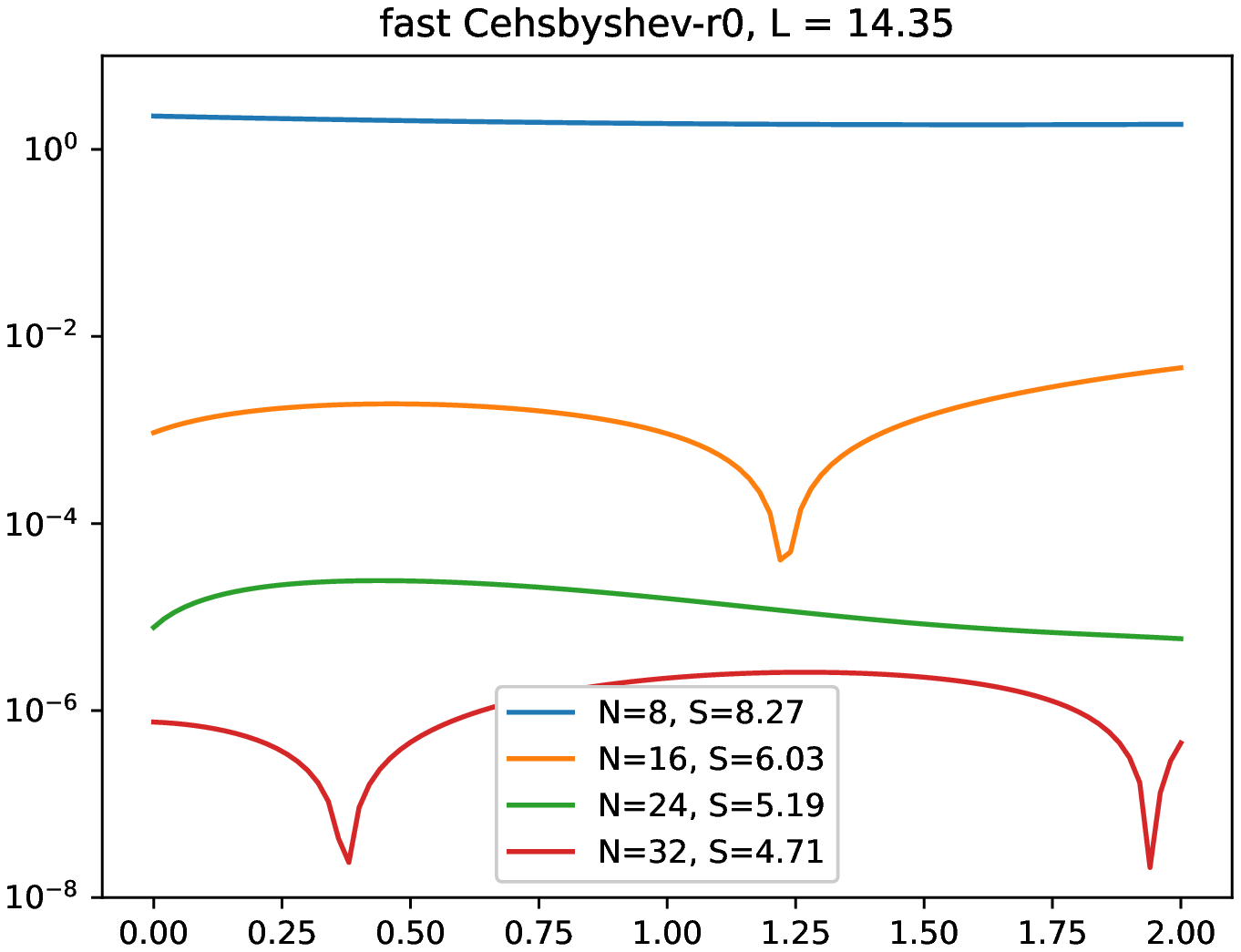}   
   \caption{(2D moments) 
   The time evolution for the absolute error of the momentum flow $P_{22}$.
   Left: the fast Fourier method.
   Right: the fast Chebyshev-0 method.}
   \label{2d_P22}
\end{figure}

\begin{figure}[htp!]
   \centering
   \includegraphics[width=.49\linewidth]{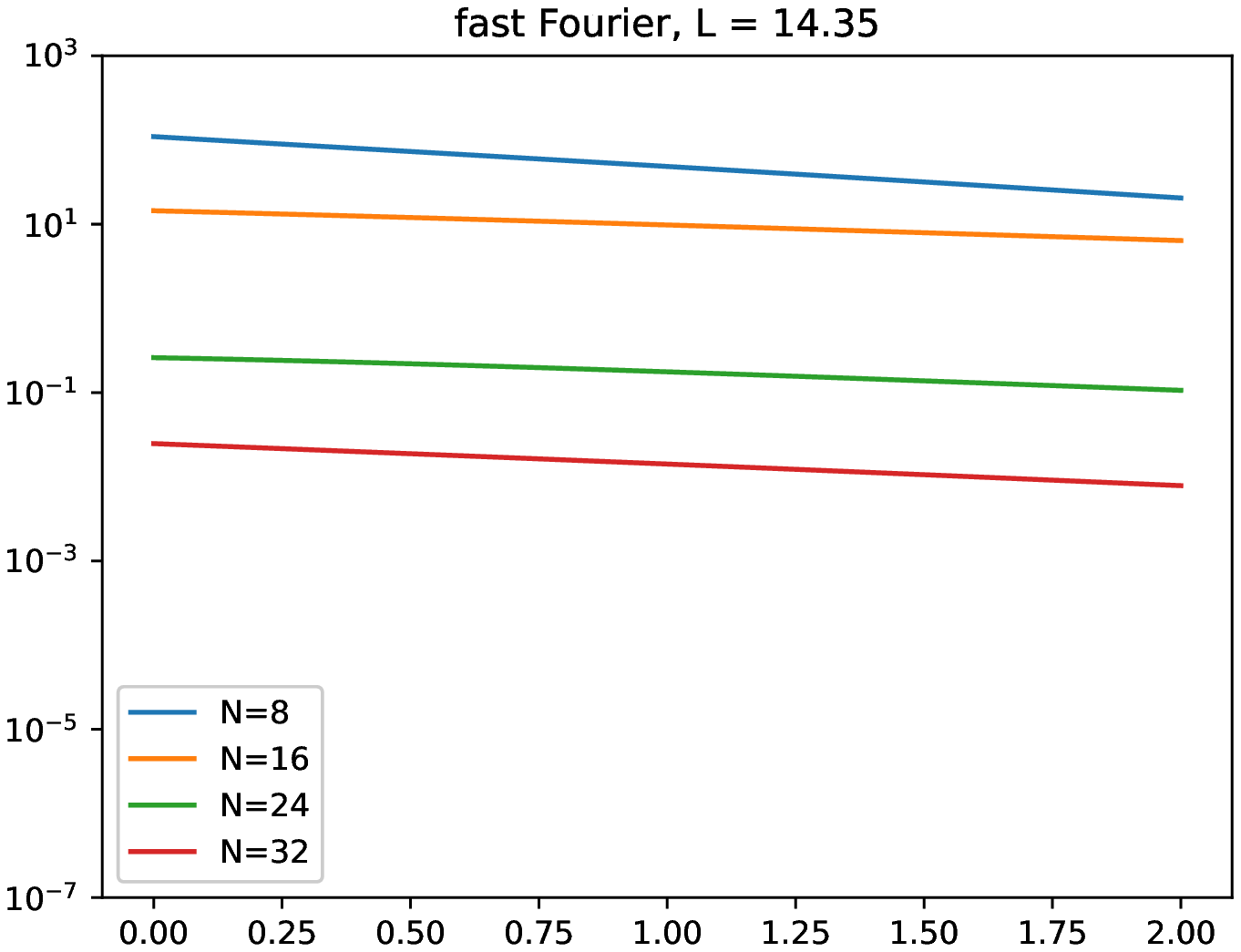}
   \includegraphics[width=.49\linewidth]{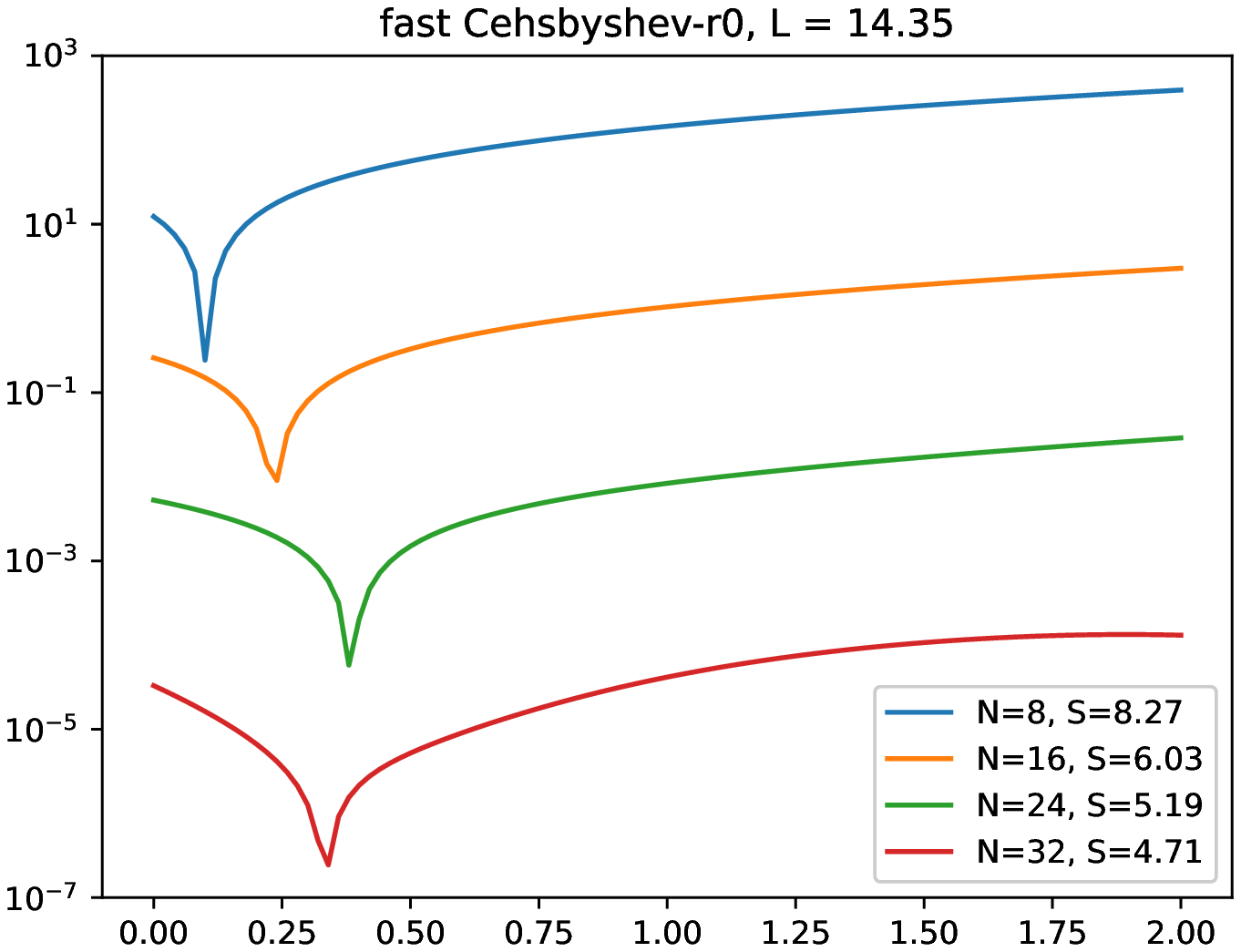}   
   \caption{(2D moments) 
   The time evolution for the absolute error of the momentum flow $q_{1}$.
   Left: the fast Fourier method.
   Right: the fast Chebyshev-0 method.}
   \label{2d_q1}
\end{figure}

\begin{figure}[htp!]
   \centering
     \includegraphics[width=.49\linewidth]{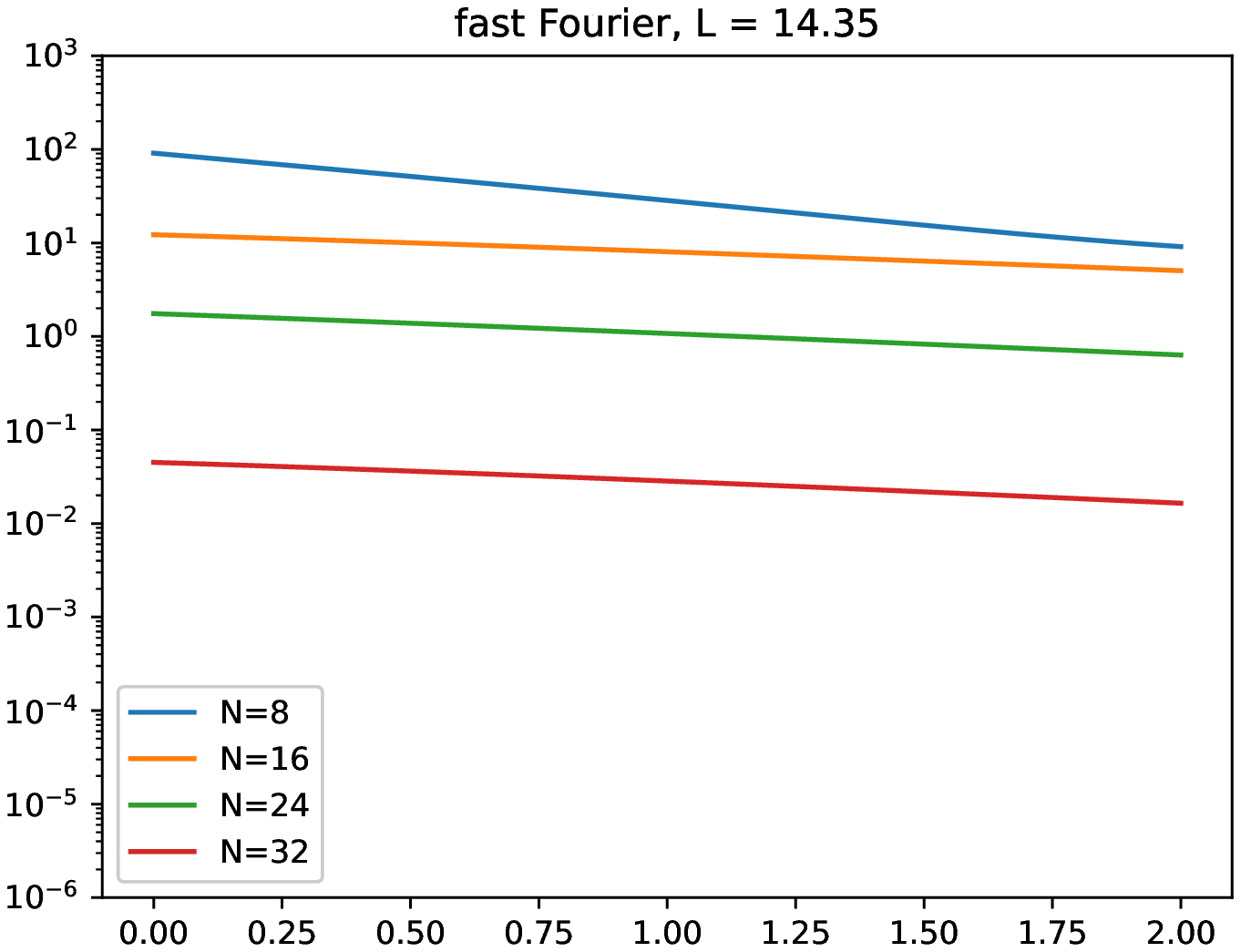}
   \includegraphics[width=.49\linewidth]{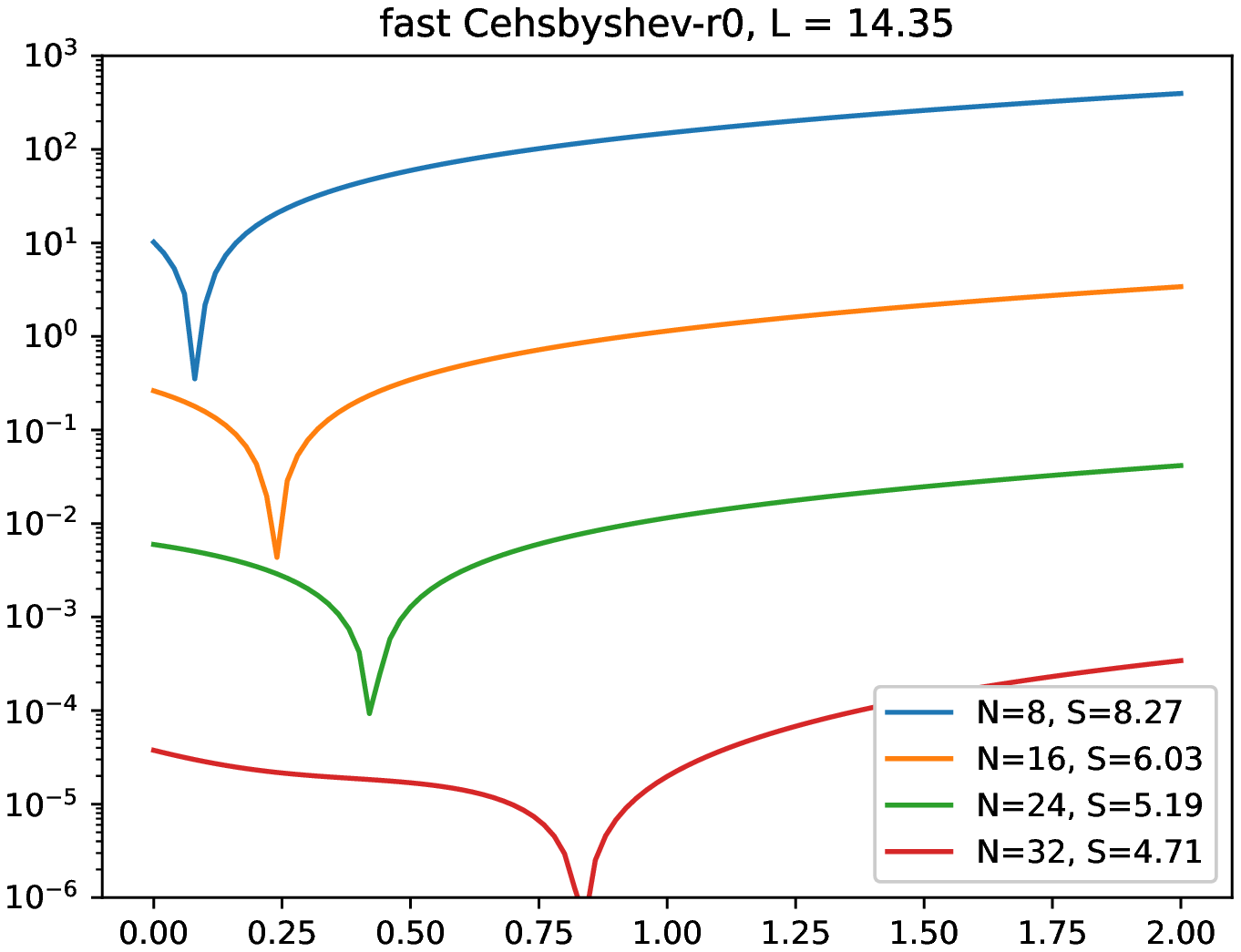}   
   \caption{(2D moments) 
    The time evolution for the absolute error of the momentum flow $q_{2}$.
   Left: the fast Fourier method.
   Right: the fast Chebyshev-0 method. }
   \label{2d_q2}
\end{figure}

\subsection{3D BKW solution}

We finally consider the 3D BKW solution. When $d=3$ and the collision kernel $\mathcal{B} \equiv 1/(4\pi)$, the following is a solution to the initial value problem (\ref{MultiDBoltz}):
\be
   f_{\text{BKW}}(t,\bmv) = 
   \f{1}{2(2\pi K)^{3/2}} \exp\left(-\f{\bmv^2}{2K}\right)
   \left(\f{5K-3}{K} + \f{1-K}{K^2} \bmv^2\right),
\ee
where $K=1-\exp(-t/6)$. As in 2D, we can obtain the exact collision operator as
\be
   Q_{\text{BKW}}(f)
   = \left\{ \left(-\f{3}{2K} + \f{\bmv^2}{2K^2}\right)f_{\text{BKW}} + 
   \f{1}{2(2 \pi K)^{3/2}} \exp\left(-\f{\bmv^2}{2K}\right)
   \left(\f{3}{K^2} + \f{K-2}{K^3} \bmv^2\right)
   \right\} K',
\ee
with $K' = \exp(-t/6)/6$.

Here we again compare the fast Fourier spectral method with the fast Chebyshev-0 method. In the former, we take domain $L=6.62$, $N_\rho = N$ quadrature points in the radial direction and $M_\bmsg=38$ Lebedev quadrature points on the unit sphere. In the latter, we choose $S$ adaptively based on $L$, $M_\bmv = N+2$ quadrature points for each dimension of $(\bmv, \bmvs)$ and $M_{\bmsg}$ Lebedev quadrature points on the unit sphere. The precision in NUFFT is selected as $\ep = 1e-14$. The $L^{\infty}$ error of $Q_{\rm BKW}(f)$ is estimated on a $30 \times 30 \times 30$ uniform grid in the rectangular domain $ [-6.3, 6.3]^3$ at time $t=6.5$.

\begin{table}[H]
   \begin{center}
   \begin{tabular}{ c|c|l } 
      \hline
      & fast Fourier $(M_{\bmsg} = 38)$  & fast Chebyshev-0  
      \\  \hline
   $N=12$  
   & 2.36e-03 & 1.61e-02 ($M_{\bmsg}=14$)  
   \\ \hline
   $N=16$ 
   & 4.37e-04 & 2.72e-03 ($M_{\bmsg}=38$) 
   \\ \hline
   $N=20$ 
   & 3.62e-05 & 3.08e-06 ($M_{\bmsg}=86$) 
   \\ \hline
   $N=24$  
   & 3.61e-06 & 3.10e-08 ($M_{\bmsg}=146$) 
   \\ \hline
   $N=28$  
   & 1.64e-07 & 1.58e-08 ($M_{\bmsg}=170$) 
   \\ \hline
   $N=32$  
   & 3.82e-08 & 7.16e-10 ($M_{\bmsg}=230$)
   \\ \hline
   \end{tabular}
   \caption{(3D BKW) The $L^\ift$ error of $Q_{\text{BKW}}(f)$ at time $t=6.5$.}
   \label{3DBKW}
   \end{center}
\end{table}

The results are reported in Table~\ref{3DBKW}. Unlike the Fourier method for which $M_{\bmsg} = 38$ is enough (we have tested that larger value  of $M_{\bmsg}$ would not further increase the accuracy), we observe that more quadrature points on the sphere are needed to get the best accuracy in the Chebyshev method. As soon as $N\geq 20$, the Chebyshev method can always obtain better accuracy than the Fourier method.

\section{Conclusion}
\label{Sec:con}

We introduced a Petrov-Galerkin spectral method for the spatially homogeneous Boltzmann equation in multi-dimensions. The mapped Chebyshev functions in $\mathbb{R}^d$ were carefully chosen to serve as the trial functions and test functions in the approximation. In the case of the algebraic mapping, we established a consistency result for approximation of the collision operator as well as the conservation property for the moments. Thanks to the close relation between the Chebyshev functions and the Fourier cosine series, we proposed a fast algorithm to alleviate the memory constraint in the precomputation and accelerate the online computation in the direct implementation. Through a series of numerical examples in 2D and 3D, we demonstrated that the proposed method can provide better accuracy (at least one or two digits for small $N$) in comparison to the popular Fourier spectral method.

\begin{appendix}

\end{appendix}

\bibliographystyle{acm}
\bibliography{reference,hu_bibtex}

\begin{thebibliography}{10}

\bibitem{AGT18}
{\sc Alonso, R., Gamba, I., and Tharkabhushanam, S.}
\newblock Convergence and error estimates for the {L}agrangian-based
  conservative spectral method for {B}oltzmann equations.
\newblock {\em SIAM J. Numer. Anal. 56\/} (2018), 3534--3579.

\bibitem{barnett2019parallel}
{\sc Barnett, A.~H., Magland, J., and af~Klinteberg, L.}
\newblock A parallel nonuniform fast {Fourier} transform library based on an
  ``exponential of semicircle" kernel.
\newblock {\em SIAM Journal on Scientific Computing 41}, 5 (2019), C479--C504.

\bibitem{Bird}
{\sc Bird, G.~A.}
\newblock {\em Molecular {G}as {D}ynamics and the {D}irect {S}imulation of
  {G}as {F}lows}.
\newblock Clarendon Press, Oxford, 1994.

\bibitem{bird1994molecular}
{\sc Bird, G.~A., and Brady, J.}
\newblock {\em Molecular gas dynamics and the direct simulation of gas flows},
  vol.~42.
\newblock Clarendon press Oxford, 1994.

\bibitem{BL}
{\sc Birdsall, C.~K., and Langdon, A.~B.}
\newblock {\em Plama Physics via Computer Simulation}.
\newblock CRC Press, 2018.

\bibitem{Cercignani}
{\sc Cercignani, C.}
\newblock {\em The {B}oltzmann {E}quation and {I}ts {A}pplications}.
\newblock Springer-Verlag, New York, 1988.

\bibitem{Cercignani00}
{\sc Cercignani, C.}
\newblock {\em Rarefied Gas Dynamics: From Basic Concepts to Actual
  Calculations}.
\newblock Cambridge University Press, Cambridge, 2000.

\bibitem{Chandrasekhar}
{\sc Chandrasekhar, S.}
\newblock {\em Radiative Transfer}.
\newblock Dover Publications, 1960.

\bibitem{Pareschi}
{\sc Dimarco, G., and Pareschi, L.}
\newblock Numerical methods for kinetic equations.
\newblock {\em Acta Numer. 23\/} (2014), 369--520.

\bibitem{FM11}
{\sc Filbet, F., and Mouhot, C.}
\newblock {Analysis of spectral methods for the homogeneous Boltzmann
  equation}.
\newblock {\em Trans. Amer. Math. Soc. 363\/} (2011), 1947--1980.

\bibitem{fonn2014polar}
{\sc Fonn, E., Grohs, P., and Hiptmair, R.}
\newblock Polar spectral scheme for the spatially homogeneous {Boltzmann}
  equation.
\newblock {\em Research Report 2014\/} (2014).

\bibitem{GHHH17}
{\sc Gamba, I., Haack, J., Hauck, C., and Hu, J.}
\newblock A fast spectral method for the {B}oltzmann collision operator with
  general collision kernels.
\newblock {\em SIAM J. Sci. Comput. 39\/} (2017), B658--B674.

\bibitem{GT09}
{\sc Gamba, I., and Tharkabhushanam, S.}
\newblock Spectral-{L}agrangian methods for collisional models of
  non-equilibrium statistical states.
\newblock {\em J. Comput. Phys. 228\/} (2009), 2012--2036.

\bibitem{gamba2017fast}
{\sc Gamba, I.~M., Haack, J.~R., Hauck, C.~D., and Hu, J.}
\newblock A fast spectral method for the {Boltzmann} collision operator with
  general collision kernels.
\newblock {\em SIAM Journal on Scientific Computing 39}, 4 (2017), B658--B674.

\bibitem{gamba2018galerkin}
{\sc Gamba, I.~M., and Rjasanow, S.}
\newblock {Galerkin}--{Petrov} approach for the {Boltzmann} equation.
\newblock {\em Journal of Computational Physics 366\/} (2018), 341--365.

\bibitem{Giovangigli}
{\sc Giovangigli, V.}
\newblock {\em Multicomponent Flow Modeling}.
\newblock Springer Science \& Business Media, 1999.

\bibitem{HQY21}
{\sc Hu, J., Qi, K., and Yang, T.}
\newblock {A new stability and convergence proof of the Fourier-Galerkin
  spectral method for the spatially homogeneous Boltzmann equation}.
\newblock {\em SIAM J. Numer. Anal. 59}, 2 (2021), 613--633.

\bibitem{hu2020petrov}
{\sc Hu, J., Shen, J., and Wang, Y.}
\newblock A {Petrov}-{Galerkin} spectral method for the inelastic {Boltzmann}
  equation using mapped {Chebyshev} functions.
\newblock {\em Kinetic \& Related Models 13}, 4 (2020).

\bibitem{HC20}
{\sc Hu, Z., and Cai, Z.}
\newblock {Burnett spectral method for high-speed rarefied gas flows}.
\newblock {\em SIAM J. Sci. Comput. 42\/} (2020), B1193--B1226.

\bibitem{HCW20}
{\sc Hu, Z., Cai, Z., and Wang, Y.}
\newblock {Numerical simulation of microflows using Hermite spectral methods}.
\newblock {\em SIAM J. Sci. Comput. 42\/} (2020), B105--B134.

\bibitem{kitzler2019polynomial}
{\sc Kitzler, G., and Sch\"{o}berl, J.}
\newblock A polynomial spectral method for the spatially homogeneous
  {Boltzmann} equation.
\newblock {\em SIAM Journal on Scientific Computing 41}, 1 (2019), B27--B49.

\bibitem{lebedev1976quadratures}
{\sc Lebedev, V.}
\newblock Quadratures on a sphere.
\newblock {\em USSR Computational Mathematics and Mathematical Physics 16}, 2
  (1976), 10--24.

\bibitem{MP06}
{\sc Mouhot, C., and Pareschi, L.}
\newblock Fast algorithms for computing the {B}oltzmann collision operator.
\newblock {\em Math. Comp. 75\/} (2006), 1833--1852.

\bibitem{mouhot2004regularity}
{\sc Mouhot, C., and Villani, C.}
\newblock Regularity theory for the spatially homogeneous {B}oltzmann equation
  with cut-off.
\newblock {\em Arch. Rational Mech. Anal. 173\/} (2004), 169--212.

\bibitem{NPT}
{\sc Naldi, G., Pareschi, L., and Toscani, G.}, Eds.
\newblock {\em Mathematical Modeling of Collective Behavior in Socio-Economic
  and Life Sciences}.
\newblock Birkh\"{a}user Basel, 2010.

\bibitem{Nanbu80}
{\sc Nanbu, K.}
\newblock Direct simulation scheme derived from the {B}oltzmann equation. {I}.
  {M}onocomponent gases.
\newblock {\em J. Phys. Soc. Jpn. 49\/} (1980), 2042--2049.

\bibitem{PR00}
{\sc Pareschi, L., and Russo, G.}
\newblock Numerical solution of the {B}oltzmann equation {I}: spectrally
  accurate approximation of the collision operator.
\newblock {\em SIAM J. Numer. Anal. 37\/} (2000), 1217--1245.

\bibitem{PR00stability}
{\sc Pareschi, L., and Russo, G.}
\newblock On the stability of spectral methods for the homogeneous {B}oltzmann
  equation.
\newblock {\em Transport Theory Statist. Phys. 29\/} (2000), 431--447.

\bibitem{shen2011spectral}
{\sc Shen, J., Tang, T., and Wang, L.-L.}
\newblock {\em Spectral methods: algorithms, analysis and applications},
  vol.~41.
\newblock Springer Science \& Business Media, 2011.

\bibitem{shen2010sparse}
{\sc Shen, J., and Wang, L.-L.}
\newblock Sparse spectral approximations of high-dimensional problems based on
  hyperbolic cross.
\newblock {\em SIAM Journal on Numerical Analysis 48}, 3 (2010), 1087--1109.

\bibitem{shen2014approximations}
{\sc Shen, J., Wang, L.-L., and Yu, H.}
\newblock Approximations by orthonormal mapped {Chebyshev} functions for
  higher-dimensional problems in unbounded domains.
\newblock {\em Journal of Computational and Applied Mathematics 265\/} (2014),
  264--275.

\bibitem{Villani02}
{\sc Villani, C.}
\newblock A review of mathematical topics in collisional kinetic theory.
\newblock In {\em Handbook of Mathematical Fluid Mechanics}, S.~Friedlander and
  D.~Serre, Eds., vol.~I. North-Holland, 2002, pp.~71--305.

\end{thebibliography}


\end{document}